\documentclass[oneeqnum,onetabnum,onefignum,onealgnum]{siamart251216}

\usepackage{amsfonts}
\usepackage{amssymb}
\usepackage{bm}
\usepackage{booktabs}
\usepackage{enumitem}
\usepackage{graphicx}
\graphicspath{{figures/}}
\usepackage{mathtools}
\usepackage{algpseudocode}
\usepackage{tikz}
\usetikzlibrary{arrows.meta,positioning}
\usepackage{float}
\restylefloat{algorithm}

\setlist[itemize]{leftmargin=2em}
\setlist[enumerate]{leftmargin=2.2em}

\numberwithin{theorem}{section}
\theoremstyle{plain}
\theoremheaderfont{\normalfont\itshape}
\theorembodyfont{\normalfont}
\theoremseparator{.}
\theoremsymbol{}
\newtheorem{assumption}{Assumption}[section]
\newtheorem{remark}{Remark}[section]
\crefname{assumption}{Assumption}{Assumptions}
\Crefname{assumption}{Assumption}{Assumptions}
\crefname{corollary}{Corollary}{Corollaries}
\Crefname{corollary}{Corollary}{Corollaries}
\crefname{remark}{Remark}{Remarks}
\Crefname{remark}{Remark}{Remarks}

\newcommand{\cC}{\mathcal{C}}
\newcommand{\cD}{\mathcal{D}}
\newcommand{\cF}{\mathcal{F}}
\newcommand{\cG}{\mathcal{G}}
\newcommand{\cH}{\mathcal{H}}

\newcommand{\cK}{\mathcal{K}}
\newcommand{\cR}{\mathcal{R}}

\newcommand{\cX}{\mathcal{X}}
\newcommand{\bbR}{\mathbb{R}}
\newcommand{\Id}{I}
\newcommand{\norm}[1]{\left\lVert #1 \right\rVert}
\newcommand{\Tail}{\operatorname{Tail}}
\newcommand{\Leak}{\operatorname{Leak}}
\DeclareMathOperator{\spanop}{span}

\newcommand{\TheTitle}{A Fourier-Aware Projection-Based Periodic Parareal Method for Time-Periodic Problems}
\newcommand{\TheAuthors}{C. Tan, Y. Xu, Y. Zhang, and Y. Su}

\headers{Fourier-Aware Projection-Based Parareal}{\TheAuthors}

\title{\TheTitle}
\author{
  Chenyi Tan\thanks{School of Mathematical Sciences, Fudan University,
    Shanghai, China (\email{cytan22@m.fudan.edu.cn},
    \email{xuyc23@m.fudan.edu.cn}, \email{yehaozhang23@m.fudan.edu.cn},
    \email{yfsu@fudan.edu.cn}).}
  \and
  Yuncheng Xu\footnotemark[1]
  \and
  Yehao Zhang\footnotemark[1]
  \and
  Yangfeng Su\footnotemark[1]
}

\begin{document}

\maketitle

\begin{abstract}
Time-periodic problems arise when the desired solution is a periodic steady
state rather than a transient trajectory.  The periodic \emph{parareal} algorithm with
a periodic coarse problem (PP-PC) is a periodicity-preserving parallel-in-time
approach for such problems. Projection-based
correction can accelerate convergence of both \emph{parareal} and PP-PC.  In this paper,
we propose a Fourier-aware construction of projection spaces and a new
correction scheme to further accelerate the convergence of projection-based PP-PC. We develop
a convergence analysis of projection-based PP-PC
with the discrepancy-based correction scheme for general nonlinear
time-periodic problems.  For an arbitrary orthogonal projection, we derive a
local one-step convergence estimate controlled by the
unresolved error and explicit nonlinear contributions.  A temporal
Fourier decomposition bounds the unresolved error by a tail--leak quantity,
which is small when dominant error modes are selected and their coefficients
are captured by the projection space.  For linear problems, the nonlinear
contributions vanish, yielding a globally valid one-step tail--leak convergence
estimate under weaker assumptions.
Experiments on linear and nonlinear problems show that Fourier-aware PP-PC
requires fewer outer iterations than Krylov-enhanced PP-PC.  For the linear
problems, the errors track the tail--leak bound.  For the nonlinear problems,
the experiments quantify the unresolved-error and explicit nonlinear
contributions in the local one-step estimate and show that the evaluated
tail--leak estimate follows the observed decay.
\end{abstract}



\section{Introduction}
\label{sec:introduction}

Time-periodic problems arise naturally in computational models with periodic
forcing or cyclic operation, where the quantity of interest is a periodic steady
state rather than a transient trajectory.  Examples include eddy-current
simulations, cyclic chemical processes, and fluid--structure
interaction models
\cite{bachinger2006eddyCurrent,hessenthaler2022timePeriodicMgrit,
vanDeRotten2006periodicReactors,vanNoorden2003cyclicProcesses}.  A typical finite-dimensional formulation
seeks a function \(\mathbf{u}:\left[0,T\right]\to\bbR^d\) such that
\begin{equation}
  \label{eq:intro-periodic-problem}
  \mathbf{u}'\left(t\right)
  =
  \mathbf{f}\left(\mathbf{u}\left(t\right),t\right),
  \qquad t\in \left[0,T\right], \qquad
  \mathbf{u}\left(0\right)=\mathbf{u}\left(T\right),
\end{equation}
where \(\mathbf{f}:\bbR^d\times\bbR\to\bbR^d\) satisfies
\(\mathbf f(\mathbf x,t+T)=\mathbf f(\mathbf x,t)\) for all
\((\mathbf x,t)\in\bbR^d\times\bbR\).

The periodic condition in \eqref{eq:intro-periodic-problem} has been treated by
several classes of numerical methods.  Frequency-domain formulations exploit
periodicity directly, from the classical Galerkin analysis of
\cite{urabe1965galerkin} to harmonic-balance methods for periodic flows
\cite{hall2002harmonicBalance}.  Multigrid methods solve the
periodicity-constrained space-time system for time-periodic parabolic problems
\cite{hackbusch1981periodicMultigrid}, and waveform-relaxation methods provide an
iterative alternative \cite{vandewalle1993dynamicIteration}.  Shooting and
Newton--Picard methods compute periodic states as fixed points of the period map
\cite{lust1998newtonPicard,seydel2010bifurcation}.

In large-scale time-domain simulations, spatial parallelism eventually
saturates, leaving sequential time propagation as the bottleneck.  This has motivated parallel-in-time algorithms that introduce
concurrency across time subintervals.  Representative
parallel-in-time approaches include block boundary-value methods, \emph{parareal},
PFASST, multigrid-reduction-in-time, and all-at-once time-domain preconditioning methods
\cite{amodio2009parallelTimeOdes,emmett2012pfasst,
falgout2014parallelTimeIntegration,mcdonald2018allAtOnce,ong2020applications};
see \cite{gander2015fiftyYears} for a survey.
The \emph{parareal} algorithm realizes this idea through an iterative fine--coarse correction.  At
each iteration, the fine propagator is applied independently on the time
subintervals, and the coarse propagator updates the solution values at the coarse time points.  The
original \emph{parareal} method was introduced by Lions, Maday, and
Turinici for initial-value problems \cite{lions2001parareal}; see also
\cite{gander2007parareal} for its convergence analysis.  For time-periodic
problems, the periodic condition must be incorporated into the
\emph{parareal} iteration.  Gander, Jiang, Song, and Zhang introduced and analyzed a
periodic variant \cite{gander2013periodicParareal}: the periodic \emph{parareal} algorithm with a
periodic coarse problem (PP-PC).  PP-PC preserves periodicity on the
coarse time points at every iteration.  However, the PP-PC
iteration may converge slowly for oscillatory or wave-like dynamics.

Projection-based correction has been widely used to improve the convergence of
parallel-in-time methods.  An early time-decomposed framework
for implicit parallel-in-time integration was introduced in
\cite{farhat2003timeDecomposed}.  Projection has also been used to stabilize
\emph{parareal} iterations for first- and second-order hyperbolic systems
\cite{dai2013stable}.  For initial-value problems, the
Krylov-enhanced \emph{parareal} method of \cite{gander2008krylov} improves the
iteration through projection.
The Krylov-enhanced PP-PC method of \cite{song2024krylovpppc} extends the
Krylov subspace correction to PP-PC,
the periodic form of \emph{parareal}, for time-periodic problems.  In these
projection-based \emph{parareal} methods, the projection space determines how
fine and coarse propagation are combined, so its choice affects the
convergence behavior of the iteration.

In this paper, we introduce a Fourier-aware projection-based PP-PC method,
which combines a new correction scheme with a new construction of the
projection space.  The scheme is
discrepancy-based and differs from that of Krylov-enhanced PP-PC.
For the projection space, Krylov-enhanced PP-PC uses the full
solution-snapshot space, built from the computed solution values
over one period.
However, Fourier techniques are widely used to exploit temporal
structure in time-dependent and time-periodic solvers
\cite{gander2023paradiag,kolmbauer2012minres,kulchytska2021multiharmonic,mcdonald2018allAtOnce}.
We develop a Fourier-aware construction of the projection
space that exploits this structure.  The construction represents coarse-time histories in the Fourier
basis associated with the coarse-time index.  A mode-selection rule specifies
which temporal modes to retain.  The corresponding modal components are then
combined into a mixed Fourier space, which serves as the projection space.

We analyze the convergence of projection-based PP-PC
with the discrepancy-based correction scheme for general nonlinear
time-periodic problems.  The starting point is a local one-step error
estimate that holds for a discrepancy-form propagator and an arbitrary
projection space, under local smoothness, nondegeneracy, and smallness
assumptions.  The estimate shows that the error after one
iteration is bounded in terms of the unresolved error and two explicit
nonlinear contributions.  When this unresolved part is represented in temporal Fourier
modes, it is bounded by a tail from modes not selected
together with a leak from selected
modes not fully captured by the projection space.  The tail and the leak are
both small when the selected set contains the error's dominant
temporal modes and the projection space captures their modal coefficients.
Our main result (\cref{thm:fourier-aware-nonlinear}) is a
Fourier-aware one-step estimate in which the tail--leak form
enters the unresolved-error term and the nonlinear remainder, while the
base-state term remains unchanged.
For linear time-periodic problems, the
nonlinear terms vanish; under weaker assumptions, the estimates then
reduce exactly to global bounds with the same constants: a
general-projection estimate and a tail--leak estimate.

The rest of this paper is organized as follows. \Cref{sec:pppc} recalls PP-PC and presents its
projection-based framework.  \Cref{sec:fourier-spaces} introduces
Fourier-aware projection-based PP-PC, its discrepancy-based correction scheme,
and its mixed Fourier space.  \Cref{sec:analysis} develops the local nonlinear
estimate, the Fourier-aware results, and the linear corollaries.
\Cref{sec:numerics} reports experiments on two linear and two nonlinear
time-periodic problems, and \cref{sec:conclusion} concludes the paper.

\section{PP-PC and a framework for its projection-based variants}
\label{sec:pppc}
In this section we first recall PP-PC and present projection-based PP-PC as a
general framework.  Then we show that Krylov-enhanced PP-PC is an instance of this
framework; our own instance is developed in \cref{sec:fourier-spaces}.

\subsection{Periodic \emph{parareal} with a periodic coarse problem}
\label{subsec:periodic-parareal}
We begin by recalling the original PP-PC method.  Let
\[
0=T_0<T_1<\cdots<T_N=T,
\qquad
T_n=n\Delta T,\qquad \Delta T=T/N .
\]
The points \(T_n\) and the intervals \(\left[T_n,T_{n+1}\right]\) are called
the coarse time points and coarse time intervals, respectively.
The solution values at the coarse time points are called the interface
values.  They are vectors in a real finite-dimensional state space, denoted by
\(\bbR^d\).
At iteration \(k\) the interface values are
\[
  \mathbf{U}^k
  =
  \left(\mathbf{U}_0^k,\ldots,\mathbf{U}_{N-1}^k\right)
  \in\left(\bbR^d\right)^N .
\]
The periodic condition identifies the endpoint value with the initial value:
\[
  \mathbf{U}_N^k=\mathbf{U}_0^k .
\]

The PP-PC method is built from a fine propagator and a coarse propagator on each
coarse time interval.
For \(0\le n<N\) and \(\mathbf{x}\in\bbR^d\), let
\[
  \cF\left(T_{n+1},T_n,\mathbf{x}\right),
  \qquad
  \cG\left(T_{n+1},T_n,\mathbf{x}\right)
\]
denote the corresponding fine and coarse propagator values over
\(\left[T_n,T_{n+1}\right]\), respectively, starting from the state \(\mathbf{x}\) at \(T_n\).  More
precisely, if \(\mathbf{u}\left(\cdot;T_n,\mathbf{x}\right)\) denotes the solution of the
underlying evolution problem satisfying
\(\mathbf{u}\left(T_n;T_n,\mathbf{x}\right)=\mathbf{x}\), then
\(\cF\left(T_{n+1},T_n,\mathbf{x}\right)\) gives an accurate approximation of
\(\mathbf{u}\left(T_{n+1};T_n,\mathbf{x}\right)\).  The coarse propagator
\(\cG\left(T_{n+1},T_n,\mathbf{x}\right)\) gives a less accurate but less expensive approximation
of the same state.  When no ambiguity arises, we write
\[
  \cF_n\left(\mathbf{x}\right):=\cF\left(T_{n+1},T_n,\mathbf{x}\right),
  \qquad
  \cG_n\left(\mathbf{x}\right):=\cG\left(T_{n+1},T_n,\mathbf{x}\right).
\]

The PP-PC iteration \cite{gander2013periodicParareal} computes
\(\mathbf{U}^{k+1}\) from \(\mathbf{U}^k\) by
\begin{equation}
  \label{eq:pppc-expanded}
  \begin{aligned}
    \mathbf{U}_0^{k+1}
    &=
    \cF_{N-1}\left(\mathbf{U}_{N-1}^k\right)
    +
    \cG_{N-1}\left(\mathbf{U}_{N-1}^{k+1}\right)
    -
    \cG_{N-1}\left(\mathbf{U}_{N-1}^k\right),\\
    \mathbf{U}_{n+1}^{k+1}
    &=
    \cF_n\left(\mathbf{U}_n^k\right)
    +
    \cG_n\left(\mathbf{U}_n^{k+1}\right)
    -
    \cG_n\left(\mathbf{U}_n^k\right),
    \qquad n=0,\ldots,N-2 .
  \end{aligned}
\end{equation}
Once \(\mathbf U^k\) is known, the fine and coarse propagator evaluations at
\(\mathbf{U}_n^k\) are available.
These fine propagator evaluations are independent across the coarse time
intervals and can therefore be performed in parallel.  The new values
\(\mathbf{U}_n^{k+1}\) remain coupled through the terms
\(\cG_n\left(\mathbf{U}_n^{k+1}\right)\).  The first equation closes the
cycle by coupling \(\mathbf{U}_0^{k+1}\) to
\(\mathbf{U}_{N-1}^{k+1}\).
This coupled solve is the periodic coarse problem in PP-PC.

\subsection{A framework for projection-based PP-PC}
\label{subsec:projection-corrected-pppc}

In PP-PC the correction in \eqref{eq:pppc-expanded} is carried by the coarse
propagator \(\cG_n\).  Projection-based PP-PC generalizes this step.  At
iteration \(k\) it constructs a \emph{projection space}
\(\mathcal V^k\subset\bbR^d\), with orthogonal projection \(P^k\) onto it, and
replaces \(\cG_n\) in the correction by a \emph{projection-based propagator}
\(\cX_{P^k,n}\).  It computes \(\mathbf U^{k+1}\) from the
coupled periodic problem
\begin{equation}
  \label{eq:projection-corrected-pppc}
  \begin{aligned}
    \mathbf{U}_0^{k+1}
    &=
    \cF_{N-1}\left(\mathbf{U}_{N-1}^k\right)
    +
    \cX_{P^k,N-1}\left(\mathbf{U}_{N-1}^{k+1}\right)
    -
    \cX_{P^k,N-1}\left(\mathbf{U}_{N-1}^k\right),\\
    \mathbf{U}_{n+1}^{k+1}
    &=
    \cF_n\left(\mathbf{U}_n^k\right)
    +
    \cX_{P^k,n}\left(\mathbf{U}_n^{k+1}\right)
    -
    \cX_{P^k,n}\left(\mathbf{U}_n^k\right),
    \qquad n=0,\ldots,N-2 .
  \end{aligned}
\end{equation}
Taking \(\cX_{P,n}=\cG_n\) recovers PP-PC \eqref{eq:pppc-expanded}.  The
framework leaves two ingredients to be specified: the projection space
\(\mathcal V^k\) and the projection-based propagator \(\cX_{P^k,n}\).  Different choices give different methods, while the structure of the
iteration is unchanged.  Once \(P^k\) and \(\cX_{P^k,n}\) are fixed, the values
\(\cF_n(\mathbf U_n^k)\) and \(\cX_{P^k,n}(\mathbf U_n^k)\) at the known iterate
are independent across the coarse intervals and are computed in parallel; the
new values \(\mathbf U_n^{k+1}\) remain coupled only through
\(\cX_{P^k,n}(\mathbf U_n^{k+1})\), and the first equation closes the cycle.
\Cref{alg:projection-corrected} summarizes the framework.

\begin{algorithm}[H]
\caption{Framework for projection-based PP-PC}
\label{alg:projection-corrected}
\begin{algorithmic}[1]
\Require Fine propagators \(\cF_n\), coarse propagators \(\cG_n\), initial
interface values \(\mathbf U^0\), a rule that constructs a projection space, and
a projection-based propagator \(\cX_{P,n}\), with \(P\) the orthogonal
projection onto the space.
\For{\(k=0,1,\ldots\)}
  \State\label{line:gen-proj} Compute the orthogonal projection \(P^k\) onto the projection space
  \(\mathcal V^k\) constructed by the rule.
  \State Compute \(\cF_n\left(\mathbf U_n^k\right)\) and
  \(\cX_{P^k,n}\left(\mathbf U_n^k\right)\), \(n=0,\ldots,N-1\), in parallel
  across the coarse intervals.
  \State\label{line:gen-solve} Solve the coupled cyclic system \eqref{eq:projection-corrected-pppc}
  for \(\mathbf U^{k+1}\).
\EndFor
\end{algorithmic}
\end{algorithm}

Krylov-enhanced PP-PC \cite{song2024krylovpppc} is an instance of this
framework.  For initial-value problems, the Krylov-enhanced \emph{parareal} method
\cite{gander2008krylov} accelerates \emph{parareal} by a projection built from solution
snapshots; the periodic version applies the same idea to PP-PC.  Its projection
space is the full solution-snapshot space
\begin{equation}
  \label{eq:krylov-snapshot-space}
  \mathcal V_{\rm K}^k
  :=
  \spanop\left\{\,\mathbf{U}_n^j:\ 0\le j\le k,\ 0\le n<N\,\right\},
\end{equation}
which in the present periodic notation collects all interface values available
up to iteration \(k\).  Let \(P_{\rm K}^k\) be the orthogonal projection onto
\(\mathcal V_{\rm K}^k\) and set \(Q_{\rm K}^k:=\Id-P_{\rm K}^k\), where \(\Id\)
is the identity matrix on \(\bbR^d\).  Its projection-based propagator is the split form
\begin{equation}
  \label{eq:krylov-enhanced-propagator}
  \cK_{{\rm K},n}^k\left(\mathbf{x}\right)
  :=
  \cF_n\left(P_{\rm K}^k\mathbf{x}\right)
  +
  \cG_n\left(Q_{\rm K}^k\mathbf{x}\right),
\end{equation}
which applies the fine propagator on the snapshot space and the coarse
propagator on its complement.  Krylov-enhanced PP-PC is therefore
\cref{alg:projection-corrected} with the projection space of
step~\ref{line:gen-proj} set to the snapshot space \(\mathcal V_{\rm K}^k\) and
the projection-based propagator of step~\ref{line:gen-solve} to the split form
\(\cK_{{\rm K},n}^k\); equivalently, the iteration
\eqref{eq:projection-corrected-pppc} with \(\cX_{P^k,n}=\cK_{{\rm K},n}^k\).

\section{Fourier-aware projection-based PP-PC}
\label{sec:fourier-spaces}

Under the framework of \cref{subsec:projection-corrected-pppc}, our method makes
two new choices: a discrepancy-based correction scheme and a Fourier-aware
projection space.  This correction scheme is realized by the discrepancy-form propagator, the precise object we define first.

\begin{definition}[Discrepancy-form propagator]
\label{def:discrepancy-form-propagator}
Let \(P:\bbR^d\to\bbR^d\) be an arbitrary orthogonal projection.  For
\(0\le n<N\), the fine--coarse discrepancy on the coarse time interval
\(\left[T_n,T_{n+1}\right]\) is defined by
\begin{equation}
  \label{eq:discrepancy-map}
  \cD_n\left(\mathbf{x}\right):=\cF_n\left(\mathbf{x}\right)-\cG_n\left(\mathbf{x}\right),
\end{equation}
and the discrepancy-form propagator associated with \(P\) is defined by
\begin{equation}
  \label{eq:discrepancy-form-propagator}
  \cH_{P,n}\left(\mathbf{x}\right)
  :=
  \cG_n\left(\mathbf{x}\right)+\cD_n\left(P\mathbf{x}\right)
  =
  \cG_n\left(\mathbf{x}\right)
  +\cF_n\left(P\mathbf{x}\right)
  -\cG_n\left(P\mathbf{x}\right).
\end{equation}
\end{definition}

The propagator \(\cH_{P,n}\) is not a coarse propagator.  It adds to the coarse
value \(\cG_n(\mathbf x)\) the fine--coarse discrepancy evaluated at the
projected state \(P\mathbf x\); projection-based PP-PC with this propagator
is \eqref{eq:projection-corrected-pppc} with \(\cX_{P^k,n}=\cH_{P^k,n}\).

The remaining choice is the projection space, the central algorithmic issue in
projection-based PP-PC.  Because the discrepancy-form propagator applies
fine--coarse information only to the projected component, a projection space is
useful to the extent that it captures the components that the iteration
still has to correct.  For time-periodic problems these components have temporal
structure that the discrete Fourier transform (DFT) in the coarse-time index
makes explicit.

At iteration \(k\), the interface values
\[
  \mathbf U^k
  =
  \left(\mathbf U_0^k,\ldots,\mathbf U_{N-1}^k\right)
\]
form the solution history over one period.  Its temporal Fourier coefficients
are
\begin{equation}
  \label{eq:temporal-fourier-coefficients}
  \widehat{\mathbf U}_\ell^k
  =
  \frac{1}{\sqrt N}
  \sum_{n=0}^{N-1}
  \mathbf U_n^k
  \exp\left(-2\pi \imath\ell n/N\right),
  \qquad
  \ell=0,\ldots,N-1.
\end{equation}
The same transform applies to any coarse-time history.
Here \(\imath\) is the imaginary unit.  The index \(\ell\) represents a temporal
Fourier mode, with phase increment \(2\pi\ell/N\) on the coarse time grid.  For
low positive indices the corresponding physical angular frequency is
\(2\pi\ell/T\); indices near \(N\) represent the negative frequencies
\(2\pi(\ell-N)/T\).  The coefficient \(\widehat{\mathbf U}_\ell^k\) is the content of the solution
history in that mode.  For notational simplicity we identify a mode with its
index \(\ell\), so that a set of temporal Fourier modes is a subset of
\(\{0,\ldots,N-1\}\).

The mode \(\ell=0\) is the mean component over the period.  Modes with
small positive indices describe slowly varying components on the cyclic coarse
time grid.  Thus the coefficients \(\widehat{\mathbf U}_\ell^k\) describe how
the solution history is distributed among temporal frequencies of the
periodic iteration.  Truncated Fourier representations have
proved effective in several time-periodic applications
\cite{bachinger2006eddyCurrent,hall2002harmonicBalance}.  The construction
exploits temporal concentration when it is present: a mode-selection rule first
chooses the temporally relevant modes (\cref{subsec:mode-selection}), and the
projection space is then built from the modal components of the solution and
the fine--coarse discrepancy histories at the selected modes
(\cref{subsec:mixed-fourier-space}).

\subsection{Selection of temporal modes}
\label{subsec:mode-selection}

The Fourier-aware construction first chooses a selected temporal mode set.
At iteration \(k\), the selected temporal mode set is denoted by
\[
  \mathcal I^k\subset\left\{0,\ldots,N-1\right\}.
\]
When only a few temporal
modes carry significant content, the construction aims to retain those modes.
Convergence may degrade if the selected set omits a mode carrying a significant
part of the current error.
The two choices below correspond to two practical regimes:
\begin{itemize}
\item The fixed choice applies when the dominant temporal frequencies of the
response are known in advance.  It then uses the same prescribed mode set at
every iteration,
\[
  \mathcal I^k=\mathcal I_{\mathrm{fix}}\quad\text{for all } k,
  \qquad
  \mathcal I_{\mathrm{fix}}\subset\left\{0,\ldots,N-1\right\},
\]
with \(\mathcal I_{\mathrm{fix}}\) the modes at those frequencies.  It is
appropriate
when the periodic response is dominated by low temporal frequencies, for which we
take \(\mathcal I_{\mathrm{fix}}\) to be the symmetric low-frequency band
\begin{equation}
  \label{eq:low-frequency-band}
  \mathcal I_L
  =
  \left\{0,1,\ldots,L\right\}\cup\left\{N-L,\ldots,N-1\right\},
  \qquad 1\le L<N/2 .
\end{equation}
This band retains the mean mode and the first \(L\) positive temporal frequencies
together with their negative-frequency counterparts.
\item The adaptive increment-energy choice applies when the dominant temporal
frequencies are not known in advance, or change during the iteration.  It uses the
current iteration increment.  For \(k\ge 1\), set
\[
  \Delta \mathbf U_n^k
  =
  \mathbf U_n^k-\mathbf U_n^{k-1}.
\]
Let \(\widehat{\Delta\mathbf U}_\ell^k\) be its temporal Fourier coefficients
and define
\begin{equation}
  \label{eq:increment-energy}
  a_\ell^k
  =
  \left\|\widehat{\Delta\mathbf U}_\ell^k\right\|_2^2,
  \qquad
  \ell=0,\ldots,N-1.
\end{equation}
Since the histories are real, \(a_\ell^k=a_{(N-\ell)\bmod N}^k\), so the
energies are properties of the conjugate groups
\(\{\ell,(N-\ell)\bmod N\}\).  The adaptive rule ranks these groups by their
combined energy and retains the \(M\) highest-ranked groups, where \(M\) is a
prescribed number.  The mode \(\ell=0\) is self-conjugate, as is \(\ell=N/2\) when \(N\) is even, so the
selected set \(\mathcal I^k\) contains at most \(2M\) modes.  For \(k=0\) no increment is available; we set
\(\mathcal I^0=\emptyset\), so the first update reduces to the PP-PC step.
The increment estimates the iteration error: for a contracting iteration,
\(\Delta\mathbf U^k\) is dominated by the error still to be removed, so its
leading temporal modes follow those of the error.  This link is
heuristic: for a noncontracting or oscillatory iteration the temporal spectrum of
\(\Delta\mathbf U^k\) need not match that of the error, and the adequacy of the
selection is assessed a posteriori in \cref{sec:numerics}.
\end{itemize}

\subsection{Construction of the Fourier-aware projection space}
\label{subsec:mixed-fourier-space}
A selected mode affects the correction only through its modal components in
the projection space.  The construction of the projection
space is guided by the role of \(P\) in
the projection-based PP-PC update \eqref{eq:projection-corrected-pppc}
and in the discrepancy-form propagator
\eqref{eq:discrepancy-form-propagator}.
For an orthogonal projection \(P\) and the current solution history \(\mathbf U^k\),
define the correction source
\[
  \mathbf c_{n+1}^k\left(P\right)
  :=
  \cF_n\left(\mathbf U_n^k\right)
  -
  \cH_{P,n}\left(\mathbf U_n^k\right),
  \qquad n=0,\ldots,N-1 .
\]
By the definition of \(\cH_{P,n}\), this source is
\[
  \mathbf c_{n+1}^k\left(P\right)
  =
  \cD_n\left(\mathbf U_n^k\right)
  -
  \cD_n\left(P\mathbf U_n^k\right).
\]
With the discrepancy-form propagator, the projection-based PP-PC step is
the cyclic system---its operator matrix understood rowwise, as made precise
below---
\begin{equation}
  \label{eq:cyclic-system}
  \begin{bmatrix}
    \Id & 0 & \cdots & -\cH_{P,N-1}\left(\cdot\right)\\
    -\cH_{P,0}\left(\cdot\right) & \Id & & 0\\
    \vdots & \ddots & \ddots & \vdots\\
    0 & \cdots & -\cH_{P,N-2}\left(\cdot\right) & \Id
  \end{bmatrix}
  \begin{bmatrix}
    \mathbf U_0^{k+1}\\
    \mathbf U_1^{k+1}\\
    \vdots\\
    \mathbf U_{N-1}^{k+1}
  \end{bmatrix}
  =
  \begin{bmatrix}
    \mathbf c_N^k\left(P\right)\\
    \mathbf c_1^k\left(P\right)\\
    \vdots\\
    \mathbf c_{N-1}^k\left(P\right)
  \end{bmatrix}.
\end{equation}
We refer to the operator on the left-hand side as the \emph{cyclic operator}
associated with \(\cH_{P,n}\), and denote it by \(\cC_P\).  Throughout, indices
of histories are read modulo \(N\), so that \(\cC_P\) acts rowwise as
\(\left(\cC_P\mathbf w\right)_{n+1}=\mathbf w_{n+1}-\cH_{P,n}(\mathbf w_n)\),
\(n=0,\ldots,N-1\).
The projection
acts at the states where the discrepancy-form propagator is evaluated and in
the correction source \(\mathbf c_{n+1}^k\left(P\right)\); the projection
space should therefore be informed by both the solution and the fine--coarse
discrepancy histories:
\begin{itemize}
\item The solution histories contain those evaluation states
	\(\mathbf U_n^m\); their selected modal components
	identify state-space directions associated with the retained
	temporal content of the periodic iterates.
\item The discrepancy histories contain the correction directions
\(\mathbf d_n^m=\cD_n\left(\mathbf U_n^m\right)\); since
\(\mathbf c_{n+1}^k\left(P\right)\) compares
\(\cD_n\left(\mathbf U_n^k\right)\) with
\(\cD_n\left(P\mathbf U_n^k\right)\), they identify the discrepancy
components the projection should capture.
\end{itemize}
The projection space is therefore built from both histories.  The
one-step estimate in \cref{sec:analysis} contains a projection-dependent source
contribution obtained by applying a local fine--coarse discrepancy bound to the
unresolved part of the current error.  A cyclic stability factor also enters
the estimate, together with two additional contributions in the nonlinear case.

For each \(0\le m\le k\) and \(\ell=0,\ldots,N-1\), the solution history
\(\mathbf U^m\) and the fine--coarse discrepancy history \(\mathbf d^m\) have
temporal Fourier coefficients
\begin{equation}
  \label{eq:solution-discrepancy-dft}
  \widehat{\mathbf U}_\ell^m
  =
  \frac{1}{\sqrt N}\sum_{n=0}^{N-1}\mathbf U_n^m
  \exp\left(-2\pi\imath\ell n/N\right),
  \quad
  \widehat{\mathbf d}_\ell^m
  =
  \frac{1}{\sqrt N}\sum_{n=0}^{N-1}\mathbf d_n^m
  \exp\left(-2\pi\imath\ell n/N\right).
\end{equation}
These coefficients are complex in general, whereas the projection space is a
subspace of \(\bbR^d\); we therefore take their real and imaginary parts,
omitting any that vanish.
The selected modal components of the solution history span the
solution-generated Fourier space
\begin{equation}
  \label{eq:solution-generated-fourier-space}
  \mathcal V_{\rm sol}^k\left(\mathcal I^k\right)
  :=
  \spanop\left\{
  \operatorname{Re}\widehat{\mathbf U}_\ell^m,
  \operatorname{Im}\widehat{\mathbf U}_\ell^m:
  \ell\in\mathcal I^k,\ 0\le m\le k
  \right\}
\end{equation}
and those of the fine--coarse discrepancy history span the
discrepancy-generated Fourier space
\begin{equation}
  \label{eq:discrepancy-generated-fourier-space}
  \mathcal V_{\rm disc}^k\left(\mathcal I^k\right)
  :=
  \spanop\left\{
  \operatorname{Re}\widehat{\mathbf d}_\ell^m,
  \operatorname{Im}\widehat{\mathbf d}_\ell^m:
  \ell\in\mathcal I^k,\ 0\le m\le k
  \right\}.
\end{equation}
The Fourier-aware projection space is their sum, the mixed Fourier space
\begin{equation}
  \label{eq:mixed-fourier-projection-space}
  \mathcal V_{\rm mix}^k\left(\mathcal I^k\right)
  :=
  \mathcal V_{\rm sol}^k\left(\mathcal I^k\right)
  +
  \mathcal V_{\rm disc}^k\left(\mathcal I^k\right).
\end{equation}
It is the smallest subspace containing both, so a single projection \(P\) onto
\(\mathcal V_{\rm mix}^k\left(\mathcal I^k\right)\) captures the selected modal components of the evaluation states and of the discrepancy directions entering the correction source.

\begin{remark}
	\label{rem:krylov-all-modes}
	Since the DFT \eqref{eq:temporal-fourier-coefficients} is an
	invertible linear transform in the coarse-time index and the histories are
	real, the real span of
	\(\{\operatorname{Re}\widehat{\mathbf U}_\ell^m,
	\operatorname{Im}\widehat{\mathbf U}_\ell^m:0\le\ell<N\}\) equals the span
	of \(\{\mathbf U_n^m:0\le n<N\}\) for each \(m\).  With all temporal modes
	selected, the solution-generated Fourier space is therefore exactly the
	full solution-snapshot space:
	\(\mathcal V_{\rm sol}^k(\{0,\ldots,N-1\})=\mathcal V_{\rm K}^k\).
\end{remark}

\subsection{The Fourier-aware projection-based PP-PC algorithm}
\label{subsec:fourier-aware-iteration}

We now combine these ingredients into Fourier-aware projection-based
PP-PC.  At each iteration \(k\), choose the selected temporal mode set
\(\mathcal I^k\) by one of the rules in \cref{subsec:mode-selection}.  The
projection space is the mixed Fourier space
\(\mathcal V_{\rm mix}^k\left(\mathcal I^k\right)\) of
\eqref{eq:mixed-fourier-projection-space}, and \(P_{\mathcal I^k}\) is
the orthogonal projection onto it.  The Fourier-aware
projection-based PP-PC update is the periodic solve
\eqref{eq:projection-corrected-pppc} with the discrepancy-form propagator and
\(P^k=P_{\mathcal I^k}\), namely
\begin{equation}
  \label{eq:fourier-aware-pppc}
  \begin{aligned}
    \mathbf U_0^{k+1}
    &=
    \cF_{N-1}\left(\mathbf U_{N-1}^k\right)
    +
    \cH_{P_{\mathcal I^k},N-1}\left(\mathbf U_{N-1}^{k+1}\right)
    -
    \cH_{P_{\mathcal I^k},N-1}\left(\mathbf U_{N-1}^k\right),\\
    \mathbf U_{n+1}^{k+1}
    &=
    \cF_n\left(\mathbf U_n^k\right)
    +
    \cH_{P_{\mathcal I^k},n}\left(\mathbf U_n^{k+1}\right)
    -
    \cH_{P_{\mathcal I^k},n}\left(\mathbf U_n^k\right),
    \qquad n=0,\ldots,N-2 .
  \end{aligned}
\end{equation}
\Cref{alg:fourier-aware-pppc} summarizes the procedure.  Its
main parameter is the band width \(L\) of the fixed rule or the mode count \(M\)
of the adaptive rule; the experimental values are given in
\cref{sec:numerics}.  In an implementation, past Fourier coefficients are
stored, so only those of the newly generated solution and fine--coarse
discrepancy histories are computed at iteration \(k\).

\begin{algorithm}[t!]
\caption{Fourier-aware projection-based PP-PC}
\label{alg:fourier-aware-pppc}
\begin{algorithmic}[1]
\Require Fine propagators \(\cF_n\), coarse propagators \(\cG_n\), initial
interface values \(\mathbf U^0\).
\For{\(k=0,1,\ldots\)}
  \State For \(n=0{:}N{-}1\), compute the fine values \(\cF_n\left(\mathbf U_n^k\right)\)
  and the coarse values \(\cG_n\left(\mathbf U_n^k\right)\) in parallel across the
  coarse intervals.
  \State\label{line:disc-dft}Compute the fine--coarse discrepancy
  \(\mathbf d_n^k=\cD_n\left(\mathbf U_n^k\right)\) by \eqref{eq:discrepancy-map}
  and the new temporal Fourier coefficients
  \(\widehat{\mathbf U}_\ell^k\) and \(\widehat{\mathbf d}_\ell^k\) by
  \eqref{eq:solution-discrepancy-dft}; reuse the stored coefficients for \(m<k\).
  \State\label{line:modes}Choose the selected temporal mode set \(\mathcal I^k\) by the
  fixed rule \eqref{eq:low-frequency-band} or the adaptive rule
  \eqref{eq:increment-energy}.
  \State\label{line:proj}Compute the orthogonal projection \(P_{\mathcal I^k}\) onto the mixed
  Fourier space \(\mathcal V_{\rm mix}^k\left(\mathcal I^k\right)\)
  \eqref{eq:mixed-fourier-projection-space}.
  \State For \(n=0{:}N{-}1\), compute \(\cH_{P_{\mathcal I^k},n}\left(\mathbf U_n^k\right)\)
  in parallel across the coarse intervals.
  \State\label{line:solve}Solve the coupled cyclic system
  \eqref{eq:fourier-aware-pppc} for \(\mathbf U^{k+1}\).
\EndFor
\end{algorithmic}
\end{algorithm}

The discrepancy-form propagator
\(\cH_{P,n}(\mathbf x)=\cG_n(\mathbf x)+\cD_n(P\mathbf x)\) evaluates the fine
propagator at the projected state \(P\mathbf x\).  In the corrected update
\eqref{eq:fourier-aware-pppc} this evaluation falls on the next iterate
\(\mathbf U^{k+1}\), the unknown of the cyclic system \eqref{eq:cyclic-system},
so the cost of the correction depends on how that system is solved.
For affine \(\cF_n\) and \(\cG_n\) the cyclic system is
linear.  Its operator is assembled from the fine propagator applied to the
projection basis, and the correction is a single linear solve with no further
fine propagations; the fine propagator is otherwise evaluated only at the current
iterate \(\mathbf U^k\), in parallel across the coarse intervals as in the
original PP-PC.

\begin{remark}[Distinction from Krylov-enhanced PP-PC]
\label{rem:krylov-distinction}
Fourier-aware PP-PC and Krylov-enhanced PP-PC are two instances of the
projection-based framework \cref{alg:projection-corrected}; they fix its two
free ingredients---the projection space of step~\ref{line:gen-proj} and the
projection-based propagator of step~\ref{line:gen-solve}---differently.
For the projection space, Fourier-aware PP-PC builds the mixed Fourier space
\(\mathcal V_{\rm mix}^k\left(\mathcal I^k\right)\)
\eqref{eq:mixed-fourier-projection-space}, whereas Krylov-enhanced
PP-PC uses the full solution-snapshot space \(\mathcal V_{\rm K}^k\)
\eqref{eq:krylov-snapshot-space}.
For the propagator, the difference is more fundamental and independent of the
space choice: Krylov-enhanced PP-PC uses the split-form propagator
\[
  \cF_n\left(P\mathbf{x}\right)
  +
  \cG_n\left(\left(\Id-P\right)\mathbf{x}\right),
\]
whereas Fourier-aware PP-PC uses the
discrepancy-form propagator
\[
  \cH_{P,n}\left(\mathbf{x}\right)
  =
  \cG_n\left(\mathbf{x}\right)
  +
  \left(\cF_n-\cG_n\right)\left(P\mathbf{x}\right).
\]
In the affine case the
two forms differ only by a constant: if \(\cG_n\) is affine, then for every
\(\mathbf x\)
\[
\cH_{P,n}(\mathbf x)-\cF_n(P\mathbf x)-\cG_n\left(\left(\Id-P\right)\mathbf x\right)
=
\cG_n(\mathbf x)-\cG_n(P\mathbf x)-\cG_n\left(\left(\Id-P\right)\mathbf x\right)
=
-\cG_n(\mathbf 0),
\]
and this constant cancels between the two evaluations of the propagator in the
update \eqref{eq:fourier-aware-pppc}, so the corresponding iterations produce
identical iterates; the distinction is therefore active only in the nonlinear
regime.
\end{remark}

\section{Error and convergence analysis}
\label{sec:analysis}

This section analyzes the error and convergence of projection-based PP-PC
with our discrepancy-based correction scheme, realized by the discrepancy-form propagator \eqref{eq:discrepancy-form-propagator}.  The analysis considers the general nonlinear time-periodic problem
\begin{equation}
	\label{eq:nonlinear-periodic-problem}
	\mathbf u'(t)=\mathbf f(\mathbf u(t),t),
	\qquad t\in[0,T],
	\qquad \mathbf u(0)=\mathbf u(T),
\end{equation}
where \(\mathbf f(\mathbf x,t+T)=\mathbf f(\mathbf x,t)\) for all
\(\mathbf x\) in a domain containing the states considered below.

Throughout, \(\mathrm D\) denotes the Jacobian of a map with respect to the
state, the subscript \(\mathbf u\) marking the partial Jacobian of
\(\mathbf f(\mathbf u,t)\) in its state argument; \(\mathbf f\) and
\(\mathrm D_{\mathbf u}\mathbf f\) are assumed continuous in
\((\mathbf u,t)\).  Let \(\mathbf u^\star_{\mathrm c}\) denote a
\(T\)-periodic solution, assumed to exist; the subscript marks the
continuous-time solution.  Linearizing \eqref{eq:nonlinear-periodic-problem}
about \(\mathbf u^\star_{\mathrm c}\) gives the \emph{variational equation}
\begin{equation}
	\label{eq:floquet-variational}
	\dot{\mathbf v}=A_\star(t)\,\mathbf v,
	\qquad
	A_\star(t):=\mathrm D_{\mathbf u}\mathbf f\!\left(\mathbf u^\star_{\mathrm c}(t),t\right),
\end{equation}
a linear system with \(T\)-periodic coefficients.  Let its fundamental matrix
\(\Phi\) satisfy \(\dot{\Phi}=A_\star\Phi\) and \(\Phi(0)=\Id\).  Then
\(\mathbf v(T)=\Phi(T)\mathbf v(0)\), where \(\Phi(T)\) is the
\emph{monodromy matrix}, whose eigenvalues are the \emph{Floquet multipliers}
\cite{chicone2006ode,hartman2002ode,seydel2010bifurcation}.

\begin{assumption}[Nondegenerate periodic solution]
	\label{ass:nondegeneracy}
	The periodic solution \(\mathbf u^\star_{\mathrm c}\) is \emph{nondegenerate}: no
	Floquet multiplier equals \(1\), equivalently the variational equation has no
	nontrivial \(T\)-periodic solution.
\end{assumption}

Define the period map
\(\varphi_T(\mathbf x):=\mathbf u(T;\mathbf x)\), where
\(\mathbf u(0;\mathbf x)=\mathbf x\).  Since
\(\mathrm D\varphi_T(\mathbf u^\star_{\mathrm c}(0))=\Phi(T)\) and
\(\Id-\Phi(T)\) is invertible by \Cref{ass:nondegeneracy}, the implicit
function theorem implies that \(\mathbf u^\star_{\mathrm c}(0)\) is locally
unique among the fixed points of \(\varphi_T\).  If the fine period map
\(\Phi_{\cF}:=\cF_{N-1}\circ\cdots\circ\cF_0\) converges to \(\varphi_T\) in
\(C^1\) near \(\mathbf u^\star_{\mathrm c}(0)\), then for sufficiently small
fine steps \(\Phi_{\cF}\) has a locally unique fixed point
\(\mathbf u_0^\star\) near \(\mathbf u^\star_{\mathrm c}(0)\).  This fixed
point is nondegenerate:
\(\Id-\mathrm D\Phi_{\cF}(\mathbf u_0^\star)\) is invertible.
We thus define the fine periodic reference
\begin{equation}
\label{eq:fine-periodic-reference}
\mathbf u^\star
=
\left(\mathbf u_0^\star,\ldots,\mathbf u_{N-1}^\star\right)
\in
\left(\bbR^d\right)^N,
\end{equation}
by the fine-propagator interface equations
\begin{equation}
\label{eq:fine-periodic-reference-interface}
\mathbf u_{n+1}^\star=\cF_n\left(\mathbf u_n^\star\right),
\qquad n=0,\ldots,N-1,
\qquad \mathbf u_N^\star=\mathbf u_0^\star .
\end{equation}
The preceding argument gives a sufficient condition for the
existence of this discrete reference.  The estimates below assume its existence
and measure the iteration error of \(\mathbf U^k\) relative to it.
The error at the coarse time points is
\[
\mathbf e_n^k:=\mathbf u_n^\star-\mathbf U_n^k,
\qquad n=0,\ldots,N-1,
\]
its maximum-in-time size is
\[
E_\infty^k:=\max_{0\le n<N}\norm{\mathbf e_n^k}_2 ,
\]
and the error history is
\[
\mathbf e^k:=\left(\mathbf e_0^k,\ldots,\mathbf e_{N-1}^k\right)
\in\left(\bbR^d\right)^N .
\]
On histories \(\mathbf w=(\mathbf w_0,\ldots,\mathbf w_{N-1})\in(\bbR^d)^N\) we use the
\emph{maximum-in-time norm} and the operator norm it induces,
\[
\norm{\mathbf w}_\infty:=\max_{0\le n<N}\norm{\mathbf w_n}_2 ,
\qquad
\norm{S}_{\infty\to\infty}:=\sup_{\mathbf w\neq\mathbf 0}\frac{\norm{S\mathbf w}_\infty}{\norm{\mathbf w}_\infty} ,
\]
so that \(E_\infty^k=\norm{\mathbf e^k}_\infty\).
History indices are taken modulo \(N\); in particular,
\(\mathbf U_N^{k+1}=\mathbf U_0^{k+1}\),
\(\mathbf u_N^\star=\mathbf u_0^\star\), and
\(\mathbf e_N^{k+1}=\mathbf e_0^{k+1}\).  For one projection-based PP-PC
update, fix an arbitrary projection space \(\mathcal V\subset\bbR^d\), let
\(P\) be its orthogonal projection, and set \(Q:=\Id-P\).  We call
\(Q\mathbf w=(Q\mathbf w_0,\ldots,Q\mathbf w_{N-1})\in(\bbR^d)^N\)
the \emph{unresolved part} of \(\mathbf w\).  The local one-step estimate
uses three assumptions.

\begin{assumption}[Local nonlinear region]
	\label{ass:local-nonlinear-neighborhood}
	Fix an iteration index \(k\) and write \(P:=P^k\).  There
	exists a convex set
	\(\mathcal B_k\subset\bbR^d\) such that, for all
	\(n=0,\ldots,N-1\) and \(s\in[0,1]\),
	\[
	\mathbf U_n^k+s\mathbf e_n^k\in\mathcal B_k,
	\qquad
	P(\mathbf U_n^k+s\mathbf e_n^k)\in\mathcal B_k,
	\qquad
	\mathbf u_n^\star,\mathbf U_n^{k+1}\in\mathcal B_k .
	\]
	We also assume that \(\cH_{P,n}\) is well defined on
	\(\mathcal B_k\).
\end{assumption}

\begin{assumption}[Local nondegeneracy]
	\label{ass:local-nonlinear-nondegeneracy}
	Under \Cref{ass:local-nonlinear-neighborhood}, each \(\cH_{P^k,n}\) is
	continuously differentiable on an open neighborhood of
	\(\mathcal B_k\).  The \emph{linearized cyclic operator}
	\(\widehat{\cC}_{P^k}\), obtained by linearizing the cyclic operator
	\(\cC_{P^k}\) of \eqref{eq:cyclic-system} at \(\mathbf u^\star\), acts on a
	periodic history \(\mathbf w\) by
	\[
	\left(\widehat{\cC}_{P^k}\mathbf w\right)_{n+1}
	:=
	\mathbf w_{n+1}
	-
	\mathrm D\cH_{P^k,n}(\mathbf u_n^\star)\mathbf w_n,
	\quad n=0,\ldots,N-1,
	\quad \mathbf w_N=\mathbf w_0 .
	\]
	The matrix
	\[
	\widehat M_{P^k}:=\mathrm D\cH_{P^k,N-1}(\mathbf u_{N-1}^\star)\cdots
	\mathrm D\cH_{P^k,0}(\mathbf u_0^\star)
	\]
	is the monodromy of the linearized cyclic system.  We
	assume this linearization is \emph{nondegenerate}:
	\(1\notin\sigma(\widehat M_{P^k})\).  By the cyclic recurrence, this is
	equivalent to the invertibility of \(\widehat{\cC}_{P^k}\).  We set
	\begin{equation}
		\label{eq:nonlinear-Gamma-definition}
		\widehat\Gamma_k:=\norm{\widehat{\cC}_{P^k}^{-1}}_{\infty\to\infty}.
	\end{equation}
	This condition concerns the linearization of \(\cH_{P^k,n}\) at the fine
	reference and is not implied by \Cref{ass:nondegeneracy}.
\end{assumption}

\begin{assumption}[Local smoothness]
	\label{ass:local-discrepancy-smoothness}
	Under the setting of \Cref{ass:local-nonlinear-neighborhood}, suppose that
	each \(\cD_n\) is continuously Fr\'echet differentiable on an open
	neighborhood of \(\mathcal B_k\).  There exist constants
	\(\mu_k\ge 0\) and \(\nu_k\ge 0\) such that, for all
	\(\mathbf x,\mathbf y\in\mathcal B_k\) and \(n=0,\ldots,N-1\),
	\begin{equation}
		\label{eq:nonlinear-discrepancy-smoothness}
		\norm{\mathrm D\cD_n(\mathbf x)}_2\le \mu_k,
		\qquad
		\norm{\mathrm D\cD_n(\mathbf x)-\mathrm D\cD_n(\mathbf y)}_2
		\le
		\nu_k\norm{\mathbf x-\mathbf y}_2 .
	\end{equation}
	Suppose in addition that the projection-based propagators \(\cH_{P,n}\) have Lipschitz
	Jacobians on \(\mathcal B_k\): for some
	constant \(\kappa_k\ge0\),
	\begin{equation}
		\label{eq:nonlinear-corrected-curvature}
		\norm{\mathrm D\cH_{P,n}(\mathbf x)-\mathrm D\cH_{P,n}(\mathbf y)}_2
		\le
		\kappa_k\norm{\mathbf x-\mathbf y}_2,
		\qquad \mathbf x,\mathbf y\in\mathcal B_k .
	\end{equation}
	The constants \(\mu_k\), \(\nu_k\), \(\kappa_k\) are associated with the local
	set \(\mathcal B_k\), which itself depends on the iterate and the projection
	through \Cref{ass:local-nonlinear-neighborhood}.
\end{assumption}

The following additional assumption is used only in the
contractive special case discussed in
\Cref{rem:contraction-recovers-nondegeneracy}.

\begin{assumption}[Coarse-propagator contraction]
	\label{ass:propagator-contraction}
	There exists a constant \(L_G>0\), independent of the coarse time interval
	index \(n\) and of \(\Delta T\), such that the coarse propagator
	satisfies the uniform contraction bound
	\[
	\norm{\cG_n\left(\mathbf x\right)-\cG_n\left(\mathbf y\right)}_2
	\le
	\frac{1}{1+L_G\Delta T}\norm{\mathbf x-\mathbf y}_2,
	\]
	for \(n=0,\ldots,N-1\) and all \(\mathbf x,\mathbf y\in\bbR^d\).
\end{assumption}

\subsection{Projection-dependent one-step analysis}
\label{subsec:nonlinear-case}

For one update of \eqref{eq:projection-corrected-pppc} using the
discrepancy-form propagator with \(P^k=P\), define the unresolved
fine--coarse discrepancy
\begin{equation}
	\label{eq:nonlinear-unresolved-correction}
	\cR_{P,n}(\mathbf x)
	:=
	\cF_n(\mathbf x)-\cH_{P,n}(\mathbf x)
	=
	\cD_n(\mathbf x)-\cD_n(P\mathbf x).
\end{equation}
If \(E_\infty^k>0\), define the \emph{unresolved-error ratio}
\(\theta_Q^k(P)\) and the \emph{projected-error ratio} \(\theta_P^k(P)\) by
\[
\theta_Q^k(P)
:=
\frac{\max_{0\le n<N}\norm{Q\mathbf e_n^k}_2}{E_\infty^k},
\qquad
\theta_P^k(P)
:=
\frac{\max_{0\le n<N}\norm{P\mathbf e_n^k}_2}{E_\infty^k},
\]
and define the \emph{base-state weight} by
\[
U_k^\perp(P)
:=
\max_{0\le n<N}\norm{Q\mathbf U_n^k}_2 .
\]
The next lemma combines these quantities in the
projection-dependent factor \(\rho_k(P)\).

\begin{lemma}[Projection-dependent local one-step convergence estimate]
	\label{lem:local-nonlinear-onestep}
	For the nonlinear problem \eqref{eq:nonlinear-periodic-problem}, let
	\(\mathbf u^\star\) be the fine periodic reference in
	\eqref{eq:fine-periodic-reference}.  Suppose that
	\Cref{ass:local-nonlinear-neighborhood,ass:local-nonlinear-nondegeneracy,ass:local-discrepancy-smoothness}
	hold at iteration \(k\) for the orthogonal projection \(P\).
	Assume also that, for some \(R>0\), the update
	\(\mathbf U^{k+1}\) satisfies \(E_\infty^{k+1}\le R\) and the
	smallness condition
	\begin{equation}
		\label{eq:nonlinear-smallness}
		\tfrac12\widehat\Gamma_k\kappa_kR<1.
	\end{equation}
	If \(E_\infty^k=0\), then \(E_\infty^{k+1}=0\).  If \(E_\infty^k>0\),
	then the projection-based PP-PC update
	\eqref{eq:projection-corrected-pppc} using the discrepancy-form propagator
	with \(P^k=P\) satisfies
	\begin{equation}
		\label{eq:local-nonlinear-projection-estimate}
		E_\infty^{k+1}
		\le
		\rho_k(P)E_\infty^k,
	\end{equation}
	where
	\begin{equation}
		\label{eq:local-nonlinear-one-step-factor}
		\rho_k(P)
		:=
		\frac{\widehat\Gamma_k}{1-\tfrac12\widehat\Gamma_k\kappa_kR}
		\left[
		\mu_k\theta_Q^k(P)
		+
		\nu_kU_k^\perp(P)\theta_P^k(P)
		+
		\frac{\nu_k}{2}
		\theta_Q^k(P)\theta_P^k(P)E_\infty^k
		\right].
	\end{equation}
\end{lemma}

\begin{proof}
	By the update \eqref{eq:projection-corrected-pppc} using the
	discrepancy-form propagator with \(P^k=P\), the iterate
	\(\mathbf U^{k+1}\) satisfies
	\[
	\mathbf U_{n+1}^{k+1}
	=
	\cF_n(\mathbf U_n^k)
	+
	\cH_{P,n}(\mathbf U_n^{k+1})
	-
	\cH_{P,n}(\mathbf U_n^k),
	\qquad n=0,\ldots,N-1 .
	\]
	For \(n=N-1\), the equality is interpreted with
	\(\mathbf U_N^{k+1}=\mathbf U_0^{k+1}\) and
	\(\mathbf u_N^\star=\mathbf u_0^\star\), so that
	\(\mathbf e_N^{k+1}=\mathbf e_0^{k+1}\).
	Subtracting this identity from
	\eqref{eq:fine-periodic-reference-interface} and using
	\[
	\cF_n(\mathbf x)
	=
	\cH_{P,n}(\mathbf x)+\cR_{P,n}(\mathbf x),
	\]
	we obtain
	\begin{equation}
		\label{eq:nonlinear-corrected-error-equation}
		\mathbf e_{n+1}^{k+1}
		=
		\bigl[
		\cH_{P,n}(\mathbf u_n^\star)
		-
		\cH_{P,n}(\mathbf U_n^{k+1})
		\bigr]
		+
		\bigl[
		\cR_{P,n}(\mathbf u_n^\star)
		-
		\cR_{P,n}(\mathbf U_n^k)
		\bigr] .
	\end{equation}
	By \Cref{ass:local-nonlinear-nondegeneracy} each \(\cH_{P,n}\) is
	differentiable; writing \(\mathbf U_n^{k+1}=\mathbf u_n^\star-\mathbf e_n^{k+1}\),
	a Taylor expansion at the reference with the curvature bound
	\eqref{eq:nonlinear-corrected-curvature} gives
	\begin{equation}
		\label{eq:nonlinear-first-bracket-linearization}
		\cH_{P,n}(\mathbf u_n^\star)-\cH_{P,n}(\mathbf U_n^{k+1})
		=
		\mathrm D\cH_{P,n}(\mathbf u_n^\star)\,\mathbf e_n^{k+1}+\boldsymbol\eta_n,
		\qquad
		\norm{\boldsymbol\eta_n}_2\le\frac{\kappa_k}{2}\norm{\mathbf e_n^{k+1}}_2^2 .
	\end{equation}
	It remains to estimate
	\(\cR_{P,n}(\mathbf u_n^\star)-\cR_{P,n}(\mathbf U_n^k)\).  Write
	\(\mathbf U=\mathbf U_n^k\), \(\mathbf e=\mathbf e_n^k\), and
	\(\cD=\cD_n\).  For \(s\in[0,1]\), set
	\[
	\mathbf z_s:=\mathbf U+s\mathbf e,
	\qquad
	\boldsymbol\zeta_s:=P\mathbf z_s .
	\]
	Thus \(\boldsymbol\zeta_s=P\mathbf U+sP\mathbf e\).
	By \Cref{ass:local-nonlinear-neighborhood}, both paths lie in
	\(\mathcal B_k\).  Using
	\[
	\cR_{P,n}(\mathbf x)
	=
	\cD_n(\mathbf x)-\cD_n(P\mathbf x)
	\]
	and the fundamental theorem of calculus along the two paths,
	\begin{align*}
		\cR_{P,n}(\mathbf u_n^\star)-\cR_{P,n}(\mathbf U_n^k)
		&=
		\int_0^1
		\left[
		\mathrm D\cD(\mathbf z_s)\mathbf e
		-
		\mathrm D\cD(\boldsymbol\zeta_s)P\mathbf e
		\right]ds
		\\
		&=
		\int_0^1 \mathrm D\cD(\mathbf z_s)Q\mathbf e\,ds
		+
		\int_0^1
		\left[
		\mathrm D\cD(\mathbf z_s)-\mathrm D\cD(\boldsymbol\zeta_s)
		\right]
		P\mathbf e\,ds .
	\end{align*}
	The first integral is bounded by
	\(\mu_k\norm{Q\mathbf e_n^k}_2\).  For the second one,
	\eqref{eq:nonlinear-discrepancy-smoothness} gives
	\[
	\norm{
		\mathrm D\cD(\mathbf z_s)-\mathrm D\cD(\boldsymbol\zeta_s)
	}_2
	\le
	\nu_k\norm{\mathbf z_s-\boldsymbol\zeta_s}_2 .
	\]
	Moreover,
	\(
	\mathbf z_s-\boldsymbol\zeta_s
	=
	Q\mathbf U_n^k+sQ\mathbf e_n^k
	\).  Therefore
	\[
	\int_0^1\norm{\mathbf z_s-\boldsymbol\zeta_s}_2\,ds
	\le
	\norm{Q\mathbf U_n^k}_2
	+
	\frac12\norm{Q\mathbf e_n^k}_2 .
	\]
	Hence integration in \(s\) gives
	\begin{equation}
		\label{eq:nonlinear-unresolved-bound}
		\norm{
			\cR_{P,n}(\mathbf u_n^\star)-\cR_{P,n}(\mathbf U_n^k)
		}_2
		\le
		\mu_k\norm{Q\mathbf e_n^k}_2
		+
		\nu_k\norm{Q\mathbf U_n^k}_2
		\norm{P\mathbf e_n^k}_2
		+
		\frac{\nu_k}{2}
		\norm{Q\mathbf e_n^k}_2
		\norm{P\mathbf e_n^k}_2 .
	\end{equation}
	Substituting
	\eqref{eq:nonlinear-first-bracket-linearization} into
	\eqref{eq:nonlinear-corrected-error-equation} shows that the error history
	\(\mathbf e^{k+1}\) solves the cyclic system
	\[
	\widehat{\cC}_{P^k}\,\mathbf e^{k+1}=\mathbf s^k+\boldsymbol\eta,
	\qquad
	\mathbf s_{n+1}^k:=\cR_{P,n}(\mathbf u_n^\star)-\cR_{P,n}(\mathbf U_n^k),
	\]
	with indices modulo \(N\).  Here \(\mathbf s^k\) is the
	unresolved source history.  The history \(\boldsymbol\eta\) has blocks
	\((\boldsymbol\eta)_{n+1}:=\boldsymbol\eta_n\), with \(\boldsymbol\eta_n\)
	from \eqref{eq:nonlinear-first-bracket-linearization}.  By the same equation,
	\(\norm{\boldsymbol\eta}_\infty\le\tfrac{\kappa_k}{2}\left(E_\infty^{k+1}\right)^2\).
	Since \(\widehat{\cC}_{P^k}\) is invertible,
	\[
	E_\infty^{k+1}
	\le
	\widehat\Gamma_k\left(
	\norm{\mathbf s^k}_\infty+\tfrac{\kappa_k}{2}\left(E_\infty^{k+1}\right)^2
	\right),
	\]
	while \eqref{eq:nonlinear-unresolved-bound}, maximized over \(n\), bounds the
	unresolved source by
	\begin{align*}
	\norm{\mathbf s^k}_\infty
	\le{}&
	\mu_k\max_{0\le n<N}\norm{Q\mathbf e_n^k}_2
	+\nu_kU_k^\perp(P)\max_{0\le n<N}\norm{P\mathbf e_n^k}_2\\
	&+\frac{\nu_k}{2}\max_{0\le n<N}\norm{Q\mathbf e_n^k}_2\max_{0\le n<N}\norm{P\mathbf e_n^k}_2 .
	\end{align*}
	If \(E_\infty^k=0\), then \(\mathbf s^k=\mathbf 0\), so
	\(E_\infty^{k+1}\le\tfrac12\widehat\Gamma_k\kappa_kR\,E_\infty^{k+1}\) by
	\(E_\infty^{k+1}\le R\), and the smallness condition \eqref{eq:nonlinear-smallness}
	forces \(E_\infty^{k+1}=0\).  Otherwise, using
	\(\left(E_\infty^{k+1}\right)^2\le R\,E_\infty^{k+1}\) and
	\eqref{eq:nonlinear-smallness},
	\[
	\left(1-\tfrac12\widehat\Gamma_k\kappa_kR\right)E_\infty^{k+1}
	\le
	\widehat\Gamma_k\norm{\mathbf s^k}_\infty;
	\]
	dividing by \(E_\infty^k\) and using the definitions of \(\theta_Q^k(P)\),
	\(\theta_P^k(P)\), and \(U_k^\perp(P)\) gives
	\eqref{eq:local-nonlinear-projection-estimate}.
\end{proof}

\begin{remark}
	The factor \(\rho_k(P)\) is the nonlinear counterpart of the one-step
	PP-PC factor in \cite[Theorem~5.2]{gander2013periodicParareal}.
	It may vary with \(k\) because both
	the projection and the local constants may change.  The three
	bracketed contributions in \eqref{eq:local-nonlinear-one-step-factor} are the unresolved-error term
	\(\mu_k\theta_Q^k(P)\), the base-state term
	\(\nu_kU_k^\perp(P)\theta_P^k(P)\), and the
	discrepancy-curvature remainder
	\(\tfrac{\nu_k}{2}\theta_Q^k(P)\theta_P^k(P)E_\infty^k\).  If \(P=0\), then
	\[
	\theta_Q^k(0)=1,\qquad \theta_P^k(0)=0,\qquad
	\rho_k(0)=
	\frac{\widehat\Gamma_k\mu_k}
	{1-\tfrac12\widehat\Gamma_k\kappa_kR},
	\]
	the local linearized analogue of the discrepancy-Lipschitz PP-PC factor.
	If \(P=\Id\), then
	\(\theta_Q^k(\Id)=U_k^\perp(\Id)=0\), so \(\rho_k(\Id)=0\), consistently
	with \(\cH_{\Id,n}=\cF_n\).
\end{remark}

\begin{remark}
	The local assumptions are tailored to the discrepancy form, which
	evaluates \(\cG_n\) at \(\mathbf x\) and \(\cD_n\) at
	\(P\mathbf x\), both covered by \(\mathcal B_k\).
	The split form \eqref{eq:krylov-enhanced-propagator} additionally
	evaluates \(\cG_n\) at \(Q\mathbf x\), which need not belong to
	\(\mathcal B_k\) and would require regularity at those additional states.
	Since the fine propagator enters only through \(\cD_n\), no contraction
	of it is required.
\end{remark}

\begin{remark}
	\label{rem:contraction-recovers-nondegeneracy}
	Coarse-propagator contraction together with a small
	discrepancy is sufficient for the nondegeneracy condition in
	\Cref{ass:local-nonlinear-nondegeneracy}.
	Indeed, \Cref{ass:propagator-contraction,ass:local-discrepancy-smoothness} give
	\[
	\mathrm D\cH_{P,n}(\mathbf u_n^\star)
	=
	\mathrm D\cG_n(\mathbf u_n^\star)
	+
	\mathrm D\cD_n(P\mathbf u_n^\star)P,
	\qquad
	\norm{\mathrm D\cH_{P,n}(\mathbf u_n^\star)}_2
	\le
	\alpha_k:=\frac{1}{1+L_G\Delta T}+\mu_k ,
	\]
	since \(P\mathbf u_n^\star\in\mathcal B_k\)
	(\Cref{ass:local-nonlinear-neighborhood} with \(s=1\)) and \(P\) is
	nonexpansive.  If \(\alpha_k<1\), a Neumann-series argument gives
	\(\widehat\Gamma_k\le(1-\alpha_k)^{-1}\), and
	\cref{lem:local-nonlinear-onestep} yields
	\begin{equation}
		\label{eq:local-nonlinear-contraction-factor}
		E_\infty^{k+1}
		\le
		\rho_k^{\alpha}(P)E_\infty^k,
		\rho_k^{\alpha}(P)
		:=
		\frac{\mu_k\theta_Q^k(P)+\nu_kU_k^\perp(P)\theta_P^k(P)
		      +\tfrac{\nu_k}{2}\theta_Q^k(P)\theta_P^k(P)E_\infty^k}{(1-\alpha_k)-\tfrac12\kappa_kR},
	\end{equation}
	provided \(\tfrac12\kappa_kR<1-\alpha_k\), which also implies
	\eqref{eq:nonlinear-smallness}.

	When the discrepancy admits the local expansion
	\(\cD_n(\mathbf x)=\boldsymbol\psi_{n,p+1}(\mathbf x)(\Delta T)^{p+1}
	+\mathbf r_n(\mathbf x)\), with \(p\ge1\) the order of the coarse propagator
	and, uniformly in \(n\), \(\boldsymbol\psi_{n,p+1}\) bounded in \(C^{1,1}\)
	and \(\mathbf r_n=O((\Delta T)^{p+2})\) in the \(C^{1,1}\) sense, one has
	\(\mu_k,\nu_k=O((\Delta T)^{p+1})\).  If, in addition,
	\(\cG_n(\mathbf x)=\mathbf x+\Delta T\,\mathbf g_n(\mathbf x)
	+O(\Delta T^2)\) in \(C^{1,1}\), with \(\mathbf g_n\) uniformly bounded
	in \(C^{1,1}\), then
	\(\kappa_k=O(\Delta T)\) and
	\(\widehat\Gamma_k=O(1/\Delta T)\).
	Hence
	\(\tfrac12\widehat\Gamma_k\kappa_kR=O(R)\).  For a fixed sufficiently small
	\(R\), the smallness factor
	\(\tfrac12\widehat\Gamma_k\kappa_kR\) remains uniformly below 1 under
	refinement.  Moreover,
	\(\rho_k(P)=O((\Delta T)^p)\), provided
	\(E_\infty^k\) and \(U_k^\perp(P)\) remain bounded under refinement.
\end{remark}

The linear time-periodic problem
\begin{equation}
	\label{eq:linear-periodic-problem}
	\mathbf u'(t)=A(t)\mathbf u(t)+\mathbf g(t),
	\qquad t\in[0,T],
	\qquad \mathbf u(0)=\mathbf u(T),
\end{equation}
where \(A(t)\in\bbR^{d\times d}\) and
\(\mathbf g(t)\in\bbR^d\) are \(T\)-periodic, is the special case of
\eqref{eq:nonlinear-periodic-problem}
for which the fine and coarse propagators are affine maps of the initial state
on each coarse time interval \([T_n,T_{n+1}]\), as holds for
standard fixed-step integrators.

For this special case, the local smoothness conditions become
global and their constants simplify.  The
discrepancy \(\cD_n\) and the projection-based propagator \(\cH_{P,n}\) have constant
Jacobians, and one may take \(\mathcal B_k=\bbR^d\).
Writing \(F_n\) and \(G_n\) for the linear parts of \(\cF_n\) and \(\cG_n\),
the projection-based propagator has the linear part \(F_nP+G_nQ\); the
constants of \Cref{ass:local-discrepancy-smoothness} reduce to
\(\mu_k=\max_{0\le n<N}\norm{F_n-G_n}_2\) and \(\nu_k=\kappa_k=0\), and the
smallness condition \eqref{eq:nonlinear-smallness} is vacuous.
The cyclic operator \(\cC_P\) is now affine, and its
linearization \(\widehat{\cC}_P\) no longer depends on the base state: it is
the block-cyclic matrix with diagonal blocks \(\Id\) and cyclic subdiagonal
blocks \(-(F_nP+G_nQ)\).  The affine shifts cancel when two histories are
subtracted, so the error history solves the linear system
\[
	\widehat{\cC}_P\,\mathbf e^{k+1}
=
\bigl((F_n-G_n)\,Q\mathbf e_n^k\bigr)_{n+1},
\]
with indices modulo \(N\).  The update exists uniquely whenever
	\(\widehat{\cC}_P\) is invertible.  In this affine setting,
	the following assumption on the fine--coarse discrepancy \(\cD_n\)
	replaces both the local nonlinear region and local smoothness assumptions
	(\Cref{ass:local-nonlinear-neighborhood,ass:local-discrepancy-smoothness}):

\begin{assumption}[Fine--coarse discrepancy Lipschitz bound]
	\label{ass:discrepancy-lipschitz}
	There exists a constant \(C_D>0\), independent of
	\(n\) and \(\Delta T\), such that the fine--coarse discrepancy
	\(\cD_n=\cF_n-\cG_n\) satisfies
	\begin{equation}
		\label{eq:discrepancy-lipschitz-bound}
		\norm{\cD_n\left(\mathbf x\right)-\cD_n\left(\mathbf y\right)}_2
		\le
		C_D\left(\Delta T\right)^{p+1}
		\norm{\mathbf x-\mathbf y}_2,
		\quad \mathbf x,\mathbf y\in\bbR^d,\quad n=0,\ldots,N-1 ,
	\end{equation}
	where \(\Delta T\) is the coarse time interval length and \(p\) is the order of
	the coarse propagator.
\end{assumption}

Under \Cref{ass:discrepancy-lipschitz},
\(\mu_k\le C_D\left(\Delta T\right)^{p+1}\),
and the one-step factor of \cref{lem:local-nonlinear-onestep}
reduces to \(\rho_k(P)=\widehat\Gamma_k\mu_k\theta_Q^k(P)\); the admissible radius \(R\) may be taken as any bound on
\(E_\infty^{k+1}\) (such a bound exists under the assumed invertibility) and drops out of the
estimate since \(\nu_k=\kappa_k=0\).  For an arbitrary orthogonal projection,
\Cref{thm:projection-level-linear-estimate} records the resulting global estimate.

\begin{corollary}[General-projection linear convergence estimate]
	\label{thm:projection-level-linear-estimate}
	For the linear time-periodic problem \eqref{eq:linear-periodic-problem},
	assume that the fine and coarse propagators \(\cF_n\) and
	\(\cG_n\) are affine and that \Cref{ass:discrepancy-lipschitz} holds.  Let \(P\) be an arbitrary
	orthogonal projection, set \(Q:=\Id-P\), and let
	\(\widehat{\cC}_P\) be the linearized cyclic operator: the block-cyclic
	matrix with diagonal blocks \(\Id\) and cyclic subdiagonal blocks
	\(-(F_nP+G_nQ)\), with \(F_n\) and \(G_n\) the linear parts of \(\cF_n\)
	and \(\cG_n\).  Suppose \(\widehat{\cC}_P\) is invertible and let
	\begin{equation}
		\label{eq:linear-Gamma-definition}
		\Gamma
		:=
		\norm{\widehat{\cC}_P^{-1}}_{\infty\to\infty} .
	\end{equation}
	Then the error of the projection-based PP-PC update
	\eqref{eq:projection-corrected-pppc} using the discrepancy-form propagator
	with \(P^k=P\) satisfies
	\begin{equation}
		\label{eq:projection-level-linear-estimate}
		E_\infty^{k+1}
		\le
		\Gamma\,C_D\left(\Delta T\right)^{p+1}
		\max_{0\le n<N}\norm{Q\mathbf e_n^k}_2 .
	\end{equation}
\end{corollary}

\begin{proof}
	Since the propagators are affine, the projection-based PP-PC update is an
	affine cyclic system whose linear part is \(\widehat{\cC}_P\); as
	\(\widehat{\cC}_P\) is invertible, the update has a unique solution
	\(\mathbf U^{k+1}\).  Written rowwise, with
	\(\mathbf U_N^{k+1}=\mathbf U_0^{k+1}\) at \(n=N-1\),
	\[
	\mathbf U_{n+1}^{k+1}
	=
	\cF_n(\mathbf U_n^k)
	+
	\cH_{P,n}(\mathbf U_n^{k+1})
	-
	\cH_{P,n}(\mathbf U_n^k),
	\qquad n=0,\ldots,N-1 .
	\]
	Since \(\mathbf u_{n+1}^\star=\cF_n(\mathbf u_n^\star)\), subtracting this
	update from the fine periodic reference equation gives
		\begin{equation}
			\label{eq:linear-error-identity-propagator}
			\begin{aligned}
				\mathbf e_{n+1}^{k+1}
				&=
				\left[
				\cH_{P,n}(\mathbf u_n^\star)
				-
				\cH_{P,n}(\mathbf U_n^{k+1})
				\right]
				\\
				&\quad+
				\left[
				\cF_n(\mathbf u_n^\star)-\cH_{P,n}(\mathbf u_n^\star)
				\right]
				-
				\left[
				\cF_n(\mathbf U_n^k)-\cH_{P,n}(\mathbf U_n^k)
				\right] .
			\end{aligned}
		\end{equation}
	Writing \(H_{P,n}:=F_nP+G_nQ\) for the linear part of \(\cH_{P,n}\), the
	first bracket equals \(H_{P,n}\mathbf e_n^{k+1}\).  The
	remaining brackets are
	\(\cR_{P,n}(\mathbf u_n^\star)-\cR_{P,n}(\mathbf U_n^k)\).  Since
	\(\cF_n(\mathbf x)-\cH_{P,n}(\mathbf x)=\cD_n(\mathbf x)-\cD_n(P\mathbf x)\)
	and the affine shifts cancel, this difference depends only on \(Q\mathbf e_n^k\):
	\begin{align*}
		&\bigl[\cF_n(\mathbf u_n^\star)-\cH_{P,n}(\mathbf u_n^\star)\bigr]
		-
		\bigl[\cF_n(\mathbf U_n^k)-\cH_{P,n}(\mathbf U_n^k)\bigr]
		\\
		={}&
		\cD_n(Q\mathbf u_n^\star)-\cD_n(Q\mathbf U_n^k).
	\end{align*}
	With indices understood modulo \(N\), define
	\(\mathbf s_{n+1}:=\cD_n(Q\mathbf u_n^\star)-\cD_n(Q\mathbf U_n^k)\).
	Then \eqref{eq:discrepancy-lipschitz-bound} gives
	\(\norm{\mathbf s_{n+1}}_2\le C_D\left(\Delta T\right)^{p+1}
	\norm{Q\mathbf e_n^k}_2\).
	Thus \eqref{eq:linear-error-identity-propagator} reads
	\[
	\mathbf e_{n+1}^{k+1}
	=
	H_{P,n}\mathbf e_n^{k+1}
	+\mathbf s_{n+1},
	\qquad n=0,\ldots,N-1,
	\]
	with \(\mathbf e_N^{k+1}=\mathbf e_0^{k+1}\).  These \(N\) relations are the
	linear system \(\widehat{\cC}_P\mathbf e^{k+1}=\mathbf s\), with
	\(\widehat{\cC}_P\) the linearized cyclic operator of the statement.  Since
	\(\widehat{\cC}_P\) is invertible,
	\[
	E_\infty^{k+1}
	=
	\max_{0\le n<N}\norm{\mathbf e_n^{k+1}}_2
	\le
	\Gamma\max_{0\le n<N}\norm{\mathbf s_{n+1}}_2
	\le
	\Gamma\,C_D\left(\Delta T\right)^{p+1}
	\max_{0\le n<N}\norm{Q\mathbf e_n^k}_2 ,
	\]
	which is \eqref{eq:projection-level-linear-estimate}.
\end{proof}

The projection enters the next-step error bound
\eqref{eq:projection-level-linear-estimate} through two quantities: the cyclic
stability factor \(\Gamma\) and the unresolved-error norm
\(\max_{0\le n<N}\norm{Q\mathbf e_n^k}_2\).

\subsection{Fourier-aware convergence}
\label{subsec:fourier-nonlinear}

The unresolved-error quantity
\(\norm{Q\mathbf e^k}_\infty\) enters the projection-dependent
factor \(\rho_k(P)\) in \eqref{eq:local-nonlinear-one-step-factor} through
\(\theta_Q^k(P)\) and appears directly in the linear estimate
\eqref{eq:projection-level-linear-estimate}.  To bound this
quantity, consider a coarse-time history
\(\mathbf w=(\mathbf w_0,\ldots,\mathbf w_{N-1})\in(\bbR^d)^N\) with temporal
Fourier coefficients \(\widehat{\mathbf w}_\ell\) defined by
\eqref{eq:temporal-fourier-coefficients}.  For an orthogonal projection \(P\) we use the same symbol for
its complex-linear extension to \(\mathbb{C}^d\),
\(P(\mathbf a+\imath\mathbf b):=P\mathbf a+\imath P\mathbf b\), and the
Euclidean norm on \(\mathbb{C}^d\) is
\(\norm{\mathbf a+\imath\mathbf b}_2^2=\norm{\mathbf a}_2^2+\norm{\mathbf b}_2^2\).  For a selected temporal mode set
\(\mathcal I\), define the tail and the selected-mode leak by
\begin{equation}
	\label{eq:linear-tail-definition}
	\Tail_{\mathcal I}(\mathbf w):=\sum_{\ell\notin\mathcal I}\norm{\widehat{\mathbf w}_\ell}_2,
	\qquad
	\Leak_{\mathcal I}^{P}(\mathbf w):=\sum_{\ell\in\mathcal I}\norm{(\Id-P)\widehat{\mathbf w}_\ell}_2 .
\end{equation}
The corresponding maximum-in-time estimate for the unresolved part \(Q\mathbf w\) of a
history is stated next.

\begin{lemma}[Tail--leak bound for the unresolved part of a history]
	\label{lem:tail-leak-projected-history}
	Let \(\mathcal I\subset\{0,\ldots,N-1\}\) be a selected temporal mode set,
	and let \(P\) be an orthogonal projection with complement \(Q:=\Id-P\).
	Then, for every history \(\mathbf w\in(\bbR^d)^N\),
	\begin{equation}
		\label{eq:tail-leak-projected-history}
		\max_{0\le n<N}
		\norm{Q\mathbf w_n}_2
		\le
		\frac{1}{\sqrt N}
		\left(
		\Tail_{\mathcal I}(\mathbf w)
		+
		\Leak_{\mathcal I}^{P}(\mathbf w)
		\right).
	\end{equation}
\end{lemma}

\begin{proof}
	For each \(n=0,\ldots,N-1\), applying \(Q\) to the inverse
	of the DFT \eqref{eq:temporal-fourier-coefficients} gives
	\[
	Q\mathbf w_n
	=
	\frac{1}{\sqrt N}
	\sum_{\ell=0}^{N-1}
	Q\widehat{\mathbf w}_\ell
	\exp(2\pi\imath\ell n/N).
	\]
	Since \(\left|\exp(2\pi\imath\ell n/N)\right|=1\), the triangle inequality
	and the split of the sum into \(\ell\notin\mathcal I\) and
	\(\ell\in\mathcal I\) yield
	\[
	\norm{Q\mathbf w_n}_2
	\le
	\frac{1}{\sqrt N}
	\left(
	\sum_{\ell\notin\mathcal I}
	\norm{Q\widehat{\mathbf w}_\ell}_2
	+
	\sum_{\ell\in\mathcal I}
	\norm{Q\widehat{\mathbf w}_\ell}_2
	\right)
	\le
	\frac{1}{\sqrt N}
	\left(
	\Tail_{\mathcal I}(\mathbf w)
	+
	\Leak_{\mathcal I}^{P}(\mathbf w)
	\right).
	\]
	Here nonexpansiveness of \(Q\) bounds the first sum by
	\(\Tail_{\mathcal I}(\mathbf w)\), while the second sum equals
	\(\Leak_{\mathcal I}^{P}(\mathbf w)\).
	Taking the maximum over \(n\) completes the proof.
\end{proof}

For the Fourier-aware projection, let \(\mathcal I^k\) be the selected temporal
mode set and let \(P_{\mathcal I^k}\) be the orthogonal projection onto
the mixed Fourier space \(\mathcal V_{\rm mix}^k(\mathcal I^k)\).  If \(E_\infty^k>0\),
define the \emph{Fourier tail--leak ratio} \(\Theta_Q^k\) and
the \emph{projected-error ratio} \(\Theta_P^k\) by
\begin{equation}
\label{eq:fourier-aware-ratios}
\Theta_Q^k
:=
\frac{
\Tail_{\mathcal I^k}(\mathbf e^k)
+
\Leak_{\mathcal I^k}^{P_{\mathcal I^k}}(\mathbf e^k)
}{
\sqrt N\,E_\infty^k
},
\qquad
\Theta_P^k:=\theta_P^k(P_{\mathcal I^k}).
\end{equation}
\Cref{lem:tail-leak-projected-history} gives
\(\theta_Q^k(P_{\mathcal I^k})\le\Theta_Q^k\).

The next result combines
\Cref{lem:local-nonlinear-onestep} with
\Cref{lem:tail-leak-projected-history}.

\begin{theorem}[Fourier-aware local convergence estimate]
	\label{thm:fourier-aware-nonlinear}
	For the nonlinear problem
	\eqref{eq:nonlinear-periodic-problem}, let \(\mathbf u^\star\) be the fine
	periodic reference in \eqref{eq:fine-periodic-reference}, fix an iteration \(k\), and let
	\(P^k=P_{\mathcal I^k}\) be the Fourier-aware projection, with complement
	\(Q^k:=\Id-P^k\).  Suppose that
	\Cref{ass:local-nonlinear-neighborhood,ass:local-nonlinear-nondegeneracy,ass:local-discrepancy-smoothness}
	hold for iteration \(k\) and the projection \(P=P^k\), that for some
	\(R>0\) the update \(\mathbf U^{k+1}\) of \eqref{eq:fourier-aware-pppc}
	satisfies \(E_\infty^{k+1}\le R\),
	and that the smallness condition \eqref{eq:nonlinear-smallness} holds.  If
	\(E_\infty^k=0\), then \(E_\infty^{k+1}=0\).  If \(E_\infty^k>0\), then
	\begin{equation}
		\label{eq:local-nonlinear-tail-leak-estimate}
		E_\infty^{k+1}
		\le
		\rho_{k,\mathrm{F}}E_\infty^k,
	\end{equation}
	where
	\begin{equation}
		\label{eq:local-nonlinear-fourier-factor}
		\rho_{k,\mathrm{F}}
		:=
		\frac{\widehat\Gamma_k}{1-\tfrac12\widehat\Gamma_k\kappa_kR}
		\left[
		\mu_k\Theta_Q^k
		+
		\nu_kU_k^\perp(P_{\mathcal I^k})\Theta_P^k
		+
		\frac{\nu_k}{2}
		\Theta_Q^k\Theta_P^kE_\infty^k
		\right].
	\end{equation}
\end{theorem}

\begin{proof}
	\Cref{lem:local-nonlinear-onestep} applies with
	\(P=P^k\) and gives \(E_\infty^{k+1}=0\) when \(E_\infty^k=0\), and
	otherwise \(E_\infty^{k+1}\le\rho_k(P^k)E_\infty^k\) with the factor
	\eqref{eq:local-nonlinear-one-step-factor}.  By definition
	\(\theta_P^k(P^k)=\Theta_P^k\), while
	\Cref{lem:tail-leak-projected-history} applied to
	\(\mathbf w=\mathbf e^k\) gives \(\theta_Q^k(P^k)\le\Theta_Q^k\).  Since
	all coefficients multiplying these ratios in \(\rho_k(P^k)\) are
	nonnegative, \(\rho_k(P^k)\le\rho_{k,\mathrm{F}}\), and
	\eqref{eq:local-nonlinear-tail-leak-estimate} follows from
	\eqref{eq:local-nonlinear-projection-estimate}.
\end{proof}

The factor \(\rho_{k,\mathrm F}\) in
\eqref{eq:local-nonlinear-fourier-factor} contains three
bracketed contributions.
Through \(\Theta_Q^k\), the Fourier tail and leak control the unresolved-error
term \(\mu_k\Theta_Q^k\); the base-state term
\(\nu_kU_k^\perp(P_{\mathcal I^k})\Theta_P^k\) and
discrepancy-curvature remainder
\(\tfrac{\nu_k}{2}\Theta_Q^k\Theta_P^kE_\infty^k\) are the two nonlinear
bracketed contributions.  When each \(\cD_n\) is affine, one may take
\(\nu_k=0\), so only \(\mu_k\Theta_Q^k\) remains among these three
bracketed contributions, as in the linear tail--leak estimate in
\Cref{thm:linear-tail-leak-estimate}.  Since \(\Theta_Q^k\) and
\(\Theta_P^k\) depend on the unknown error, they are a posteriori diagnostics.
The factor \(\rho_{k,\mathrm F}\) may vary with \(k\) and need
not be less than one; hence \Cref{thm:fourier-aware-nonlinear} gives a one-step
estimate and does not by itself imply a uniform contraction.

\begin{corollary}[Linear tail--leak convergence estimate]
	\label{thm:linear-tail-leak-estimate}
	For the linear time-periodic problem \eqref{eq:linear-periodic-problem},
	assume that the fine and coarse propagators \(\cF_n\) and
	\(\cG_n\) are affine and that \Cref{ass:discrepancy-lipschitz} holds.  At iteration \(k\), let
	\(\mathcal I^k\) be the selected temporal
	mode set and let \(P_{\mathcal I^k}\) be the orthogonal projection onto
	the mixed Fourier space \(\mathcal V_{\rm mix}^k(\mathcal I^k)\) of
	\cref{sec:fourier-spaces}.
	Suppose the linearized cyclic operator
	\(\widehat{\cC}_{P_{\mathcal I^k}}\) of
	\Cref{thm:projection-level-linear-estimate} is invertible, and
	set
	\(\Gamma_k:=\norm{\widehat{\cC}_{P_{\mathcal I^k}}^{-1}}_{\infty\to\infty}\).
	Then the Fourier-aware projection-based
	PP-PC update satisfies
	\begin{equation}
		\label{eq:linear-tail-leak-estimate}
		E_\infty^{k+1}
		\le
		\frac{\Gamma_k\,C_D\left(\Delta T\right)^{p+1}}
		{\sqrt N}
		\left(
		\Tail_{\mathcal I^k}(\mathbf e^k)
		+
		\Leak_{\mathcal I^k}^{P_{\mathcal I^k}}(\mathbf e^k)
		\right).
	\end{equation}
\end{corollary}

\begin{proof}
	\Cref{lem:tail-leak-projected-history}, applied with
	\(\mathbf w=\mathbf e^k\) and \(P=P_{\mathcal I^k}\), bounds the
	unresolved-error term in \eqref{eq:projection-level-linear-estimate};
	\eqref{eq:linear-tail-leak-estimate} follows.
\end{proof}

This estimate separates temporal-mode selection from projection-space choice.
Selecting all temporal modes gives \(\Tail_{\mathcal I^k}(\mathbf e^k)=0\).
The equality \(\Leak_{\mathcal I^k}^{P_{\mathcal I^k}}(\mathbf e^k)=0\) holds
if and only if the projection space contains both
\(\operatorname{Re}\widehat{\mathbf e}_\ell^k\) and
\(\operatorname{Im}\widehat{\mathbf e}_\ell^k\) for every
\(\ell\in\mathcal I^k\).  The solution-generated space
\(\mathcal V_{\rm sol}^k(\mathcal I^k)\) alone need not contain these error
coefficients.  The source
\(\cR_{P,n}(\mathbf u_n^\star)-\cR_{P,n}(\mathbf U_n^k)\) in
\eqref{eq:nonlinear-corrected-error-equation} motivates enriching the
projection space by \(\mathcal V_{\rm disc}^k(\mathcal I^k)\).

\section{Numerical experiments}
\label{sec:numerics}

In this section, we present numerical experiments for Fourier-aware
projection-based PP-PC (\cref{alg:fourier-aware-pppc}; \emph{Fourier-aware
PP-PC} for short).  The experiments
serve three purposes:
\begin{enumerate}[label=\textup{(\roman*)}]
	\item to compare the outer-iteration convergence and projection-space
	dimensions of the proposed iteration with PP-PC
	\cite{gander2013periodicParareal} and Krylov-enhanced PP-PC
	\cite{song2024krylovpppc} on two linear and two nonlinear time-periodic
	problems;
	\item to evaluate the linear tail--leak estimate of
	\cref{thm:linear-tail-leak-estimate} and illustrate, with frozen projection
	spaces, the roles of the unresolved-error norm and cyclic stability factor in
	\cref{thm:projection-level-linear-estimate};
	\item to diagnose the three unweighted
	projection-dependent factors that enter the nonlinear one-step numerator in
	\cref{lem:local-nonlinear-onestep} and to evaluate the nonlinear tail--leak
	estimate of \cref{thm:fourier-aware-nonlinear} using diagnostic estimates of
	the local constants.
\end{enumerate}

Relative to the fine periodic reference
\(\mathbf u^\star\) of \cref{sec:analysis}, the reported error is
\(E_\infty^k=\max_{0\le n<N}
\norm{\mathbf u_n^\star-\mathbf U_n^k}_2\).
Each run starts from the periodic coarse solution \(\mathbf U^0\) and uses the
problem-specific outer-iteration budget \(K\).  All periodic
solves are direct in the linear tests and use damped Newton shooting in the
nonlinear tests, with a full-period residual tolerance of \(10^{-13}\).
Newton Jacobians are assembled by forward differences with componentwise
steps \(10^{-7}\max\{1,|x_j|\}\).  All experiments were carried out in MATLAB
R2024a on a laptop with an Intel Core i5-13500H CPU at 3.19 GHz and 32 GB RAM.

\subsection{Linear problems}
\label{subsec:numerics-linear}

Both problems have period \(T=1\) and homogeneous Dirichlet conditions on
\(x\in(0,1)\).  We use second-order centered differences with
\(\Delta x=1/32\); the first-order wave and realified Schr\"odinger systems
both have dimension \(d=62\).  On each coarse time interval, \(\cF\) uses
\(32\) Crank--Nicolson steps and \(\cG\) one backward Euler step.

The \emph{wave problem} is
\begin{equation}
  \label{eq:numerics-wave}
  \partial_{tt}u=\partial_{xx}u+f(x,t),
  \qquad
  f(x,t)=-\left(4\pi^2x(x-1)+2\right)\sin(2\pi t),
\end{equation}
for which a time-periodic solution is \(u(x,t)=x(x-1)\sin(2\pi t)\).  We use
\(N=16\), \(K=8\), and the fixed mode set
\(\mathcal I_1=\{0,1,N-1\}\), which contains the known frequency pair.
At \(T=1\), periodic even spatial modes make the continuous problem degenerate,
while the odd-mode forcing admits the displayed solution.  The discrete fine
period map is nevertheless nondegenerate: the semidiscrete
frequency closest to \(2\pi\) is \(64\sin(\pi/32)\approx6.273\ne2\pi\), and its nearest multiplier is about
\(10^{-2}\) from \(1\).

The \emph{forced Schr\"odinger problem} is
\begin{equation}
  \label{eq:numerics-schrodinger}
  \imath\,\partial_t\psi=-\partial_{xx}\psi+f(x,t),
\end{equation}
with periodic solution
\[
  \psi(x,t)
  =
  \phi_1(x)e^{2\pi\imath t}
  +0.35\,\phi_2(x)e^{-4\pi\imath t},
\]
where \(\phi_1(x)=x(1-x)\) and
\(\phi_2(x)=x(1-x)(1+0.25\sin\pi x)\).  Writing \(\psi_h\) for its grid
samples and \(L_h\) for the centered-difference Laplacian, we set
\(f_h=\imath\,\dot\psi_h+L_h\psi_h\), so that \(\psi_h\) solves the
semi-discrete system exactly.  We use \(N=64\), \(K=16\), and, treating the
frequencies as unknown, apply the adaptive increment-energy rule of
\cref{subsec:mode-selection} with a budget of five conjugate groups
(\(\lvert\mathcal I^k\rvert\le10\)).

\paragraph{Convergence comparison}
\Cref{fig:linear-convergence,tab:linear-summary} show that
Fourier-aware PP-PC reaches \(E_\infty^k\le10^{-10}\) first: at \(k=8\) for
the wave problem and \(k=3\) for the Schr\"odinger problem, versus \(k=15\)
and \(k=8\) for Krylov-enhanced PP-PC; PP-PC reaches neither threshold within
the respective budgets.  The wave panel extends to \(k=15\) for the later
Krylov crossing; the Schr\"odinger panel stops at \(k=7\).  The table gives each
nominal run's final error, threshold iteration, maximal dimension, and mode
count.  A dash marks no threshold hit within \(K\) or no Fourier selection, as
appropriate.  Both projection spaces use the same relative singular-value
tolerance \(10^{-12}\).

\begin{figure}[b!]
  \centering
  \includegraphics[width=0.48\textwidth]{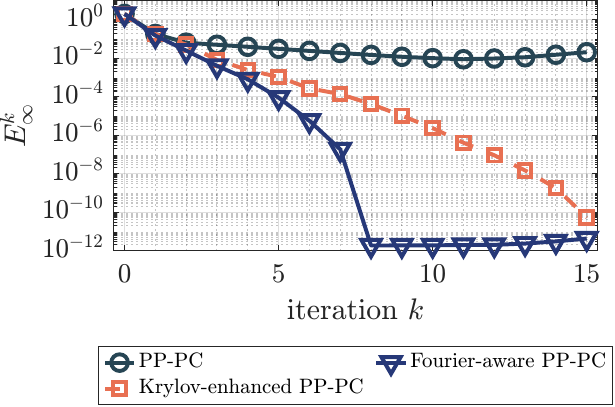}
  \hfill
  \includegraphics[width=0.48\textwidth]{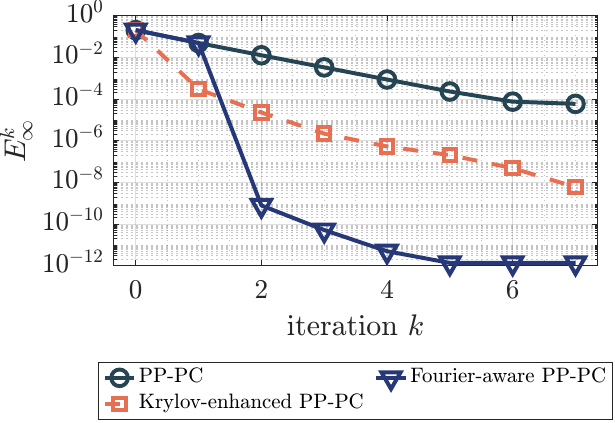}
  \caption{Error histories \(E_\infty^k\) of PP-PC, Krylov-enhanced PP-PC,
  and Fourier-aware PP-PC for the wave problem \eqref{eq:numerics-wave}
  (left, shown through \(k=15\)) and the forced
  Schr\"odinger problem \eqref{eq:numerics-schrodinger}
  (right, shown through \(k=7\)).}
  \label{fig:linear-convergence}
\end{figure}

\begin{table}[b!]
  \caption{Convergence and projection-space dimension for the two linear
  problems.}
  \label{tab:linear-summary}
  \centering
  \footnotesize
  \begin{tabular}{llcccc}
    \toprule
    Problem & Method & \(E_\infty^{K}\) &
    \(E_\infty^k\le10^{-10}\) & \(\max_k\dim\) & \(|\mathcal I^k|\) \\
    \midrule
    Wave \((K=8)\) & PP-PC            & 1.5e-02 & --  & 0   & -- \\
                   & Krylov-enhanced PP-PC & 4.2e-05 & --  & 16  & -- \\
                   & Fourier-aware PP-PC   & 1.8e-12 & 8   & 32  & 3 \\
    \midrule
    Schr\"odinger \((K=16)\) & PP-PC            & 4.4e-04 & --  & 0   & -- \\
                             & Krylov-enhanced PP-PC & 1.0e-13 & 8   & 32  & -- \\
                             & Fourier-aware PP-PC   & 1.4e-12 & 3   & 30  & \(\le\)10 \\
    \bottomrule
  \end{tabular}
\end{table}

\paragraph{Evaluation of the tail--leak estimate}

We evaluate \cref{thm:linear-tail-leak-estimate} along the Fourier-aware
iterations.  For its cyclic factor
\(\Gamma_k=\norm{\widehat{\cC}_{P_{\mathcal I^k}}^{-1}}_{\infty\to\infty}\),
we use the computable block-row majorant
\begin{equation}
\label{eq:block-row-majorant}
\Gamma_k\le\overline\Gamma_k:=\max_{0\le n<N}\sum_{m=0}^{N-1}
\norm{\bigl(\widehat{\cC}_{P_{\mathcal I^k}}^{-1}\bigr)_{nm}}_2.
\end{equation}
At the fixed discretization, the affine discrepancy has uniform Lipschitz
constant \(\delta_{FG}:=\max_n\norm{F_n-G_n}_2\), where \(F_n\) and \(G_n\)
are the linear parts of \(\cF_n\) and \(\cG_n\).  Applying the estimate with
these computable quantities gives
\begin{equation}
  \label{eq:numerics-evaluated-bound}
  B_{k+1}
  :=
  \overline\Gamma_k\,\delta_{FG}\,
  \frac{\Tail_{\mathcal I^k}\left(\mathbf e^k\right)
        +\Leak_{\mathcal I^k}^{P_{\mathcal I^k}}\left(\mathbf e^k\right)}
       {\sqrt N}.
\end{equation}
Thus \(E_\infty^{k+1}\le B_{k+1}\) in exact arithmetic.  The tail and leak use
the measured error history \(\mathbf e^k\), so \(B_{k+1}\) is a posteriori.
The estimate therefore applies to the Crank--Nicolson wave propagator without
a fine-propagator contraction assumption.

\begin{table}[b!]
  \caption{Tail--leak values and the evaluated bound for the two linear
  problems.}
  \label{tab:linear-tail-leak}
  \centering
  \begin{tabular}{llcccccc}
    \toprule
    & & \multicolumn{4}{c}{Fourier-aware}
    & \multicolumn{2}{c}{Krylov-enhanced} \\
    \cmidrule(lr){3-6}\cmidrule(lr){7-8}
    Problem & \(k\) & \(\Tail\) & \(\Leak\) & \(B_{k+1}\) & \(E_\infty^{k+1}\)
    & \(\Tail\) & \(\Leak\) \\
    \midrule
    Wave & 0 & 4.4e-13 & 3.8e-01 & 4.9e+02 & 1.5e-01 & 0 & 8.0e-01 \\
         & 2 & 5.0e-13 & 3.2e-03 & 7.0e+01 & 4.1e-03 & 0 & 1.6e-02 \\
         & 4 & 1.1e-12 & 2.7e-05 & 1.6e+00 & 9.0e-05 & 0 & 1.2e-03 \\
         & 6 & 9.0e-12 & 2.5e-08 & 1.1e-03 & 1.8e-07 & 0 & 8.6e-05 \\
         & 7 & 8.5e-12 & 7.5e-12 & 4.4e-07 & 1.8e-12 & 0 & 2.1e-05 \\
    \midrule
    Schr\"odinger & 0 & 2.3e+00 & 0 & 1.1e+01 & 5.0e-02 & 0 & 2.8e-03 \\
                  & 1 & 2.0e-13 & 7.8e-09 & 3.7e-07 & 7.9e-10 & 0 & 2.3e-04 \\
                  & 2 & 1.6e-13 & 4.5e-10 & 2.1e-08 & 5.2e-11 & 0 & 2.3e-05 \\
                  & 3 & 1.6e-13 & 3.6e-11 & 1.7e-09 & 5.0e-12 & 0 & 5.3e-06 \\
    \bottomrule
  \end{tabular}
\end{table}

\begin{figure}[b!]
  \centering
  \includegraphics[width=0.48\textwidth]{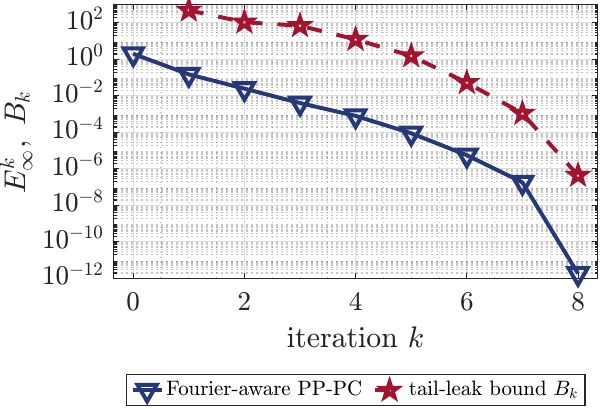}
  \hfill
  \includegraphics[width=0.48\textwidth]{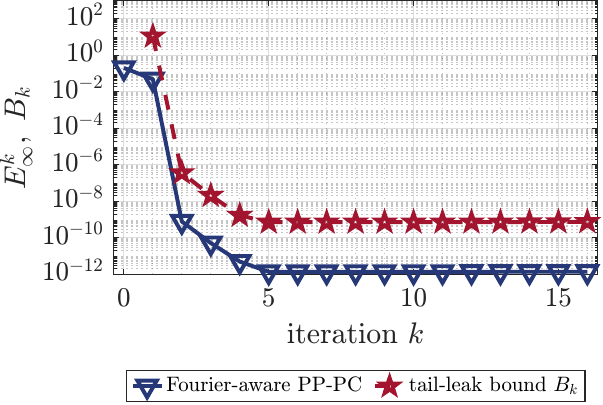}
  \caption{Measured error \(E_\infty^k\) of the Fourier-aware iterations and the
  evaluated tail--leak bound \(B_k\) of \eqref{eq:numerics-evaluated-bound}
  for the wave (left) and Schr\"odinger (right) problems; \(B_0\) is not
  defined.}
  \label{fig:linear-bound}
\end{figure}

\Cref{tab:linear-tail-leak,fig:linear-bound} show that \(B_k\) lies above
\(E_\infty^k\) at every computed step and follows the initial error decay.  The
bound is not sharp, remaining roughly two to five orders above the measured
error.  The prefactor \(\overline\Gamma_k\delta_{FG}\) contributes to this
slack.  The block-row majorant \(\overline\Gamma_k\) ranges over
\(8\times10^1\)--\(4\times10^3\) for the wave problem and
\(3\times10^1\)--\(3\times10^2\) for the Schr\"odinger problem, while
\(\delta_{FG}\approx63\) and \(1.2\), respectively.  The larger wave values
of \(\overline\Gamma_k\) are consistent with its near-resonant discretization.

The tail and leak columns separate the two contributions to the
unresolved-error bound.  For the fixed-band wave problem, the Fourier-aware
tail stays at roundoff level and the leak decays.  For the adaptive
Schr\"odinger problem, \(\mathcal I^0\) is empty, so all initial error modes
contribute to the tail.  After the first selection, the tail falls to roundoff
level and the leak decays.  For comparison, the table also
reports Krylov-enhanced diagnostics in the all-mode representation of
\cref{rem:krylov-all-modes}; its diagnostic tail is zero, but its leak need
not be.  Thus a small tail alone does not make the tail--leak bound small.
\paragraph{A posteriori test of projection dependence}

For a fixed \(P\), \cref{thm:projection-level-linear-estimate}
identifies two projection-dependent quantities in the next-step error bound:
the unresolved-error norm \(\max_{0\le n<N}\norm{Q\mathbf e_n^k}_2\), with
\(Q=\Id-P\), and the cyclic stability factor \(\Gamma\) associated with
\(\widehat{\cC}_P\).  To examine these quantities a
posteriori, we use two frozen projection spaces.  We construct \(P_F\) from
the mixed Fourier space and \(P_K\) from
the solution-snapshot space, using their respective histories through the same
iteration \(k_\ast\).  For each problem, \(k_\ast\) is the first iteration at
which Fourier-aware PP-PC satisfies \(E_\infty^{k_\ast}\le10^{-10}\).  We
restart the same discrepancy-form update \eqref{eq:projection-corrected-pppc}
from \(\mathbf U^0\) with each projection fixed, so only \(P\) differs.
\Cref{tab:linear-posterior} reports \(\max_n\norm{Q\mathbf e_n^0}_2\), the
corresponding block-row majorant \(\overline\Gamma(P)\) obtained from
\eqref{eq:block-row-majorant}, and \(E_\infty^1\); \cref{fig:linear-posterior}
shows the subsequent fixed-projection histories.

\begin{table}[b!]
  \caption{Frozen-projection diagnostics for the two linear problems.}
  \label{tab:linear-posterior}
  \centering
  \begin{tabular}{@{}llccc@{}}
    \toprule
    Problem & Projection & \(\max_n\norm{Q\mathbf e_n^0}_2\) &
    \(\overline\Gamma(P)\) & \(E_\infty^1\) \\
    \midrule
    Wave & \(P_F\) & 1.8e-12 & 1.76e+03 & 1.8e-12 \\
         & \(P_K\) & 6.7e-07 & 2.41e+03 & 1.0e-05 \\
    \midrule
    Schr\"odinger & \(P_F\) & 2.8e-12 & 3.08e+02 & 4.4e-12 \\
                  & \(P_K\) & 4.7e-07 & 3.08e+02 & 5.3e-07 \\
    \bottomrule
  \end{tabular}
\end{table}

\begin{figure}[b!]
  \centering
  \includegraphics[width=0.48\textwidth]{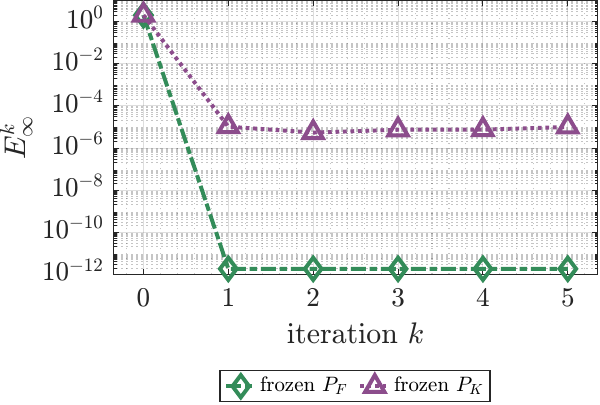}
  \hfill
  \includegraphics[width=0.48\textwidth]{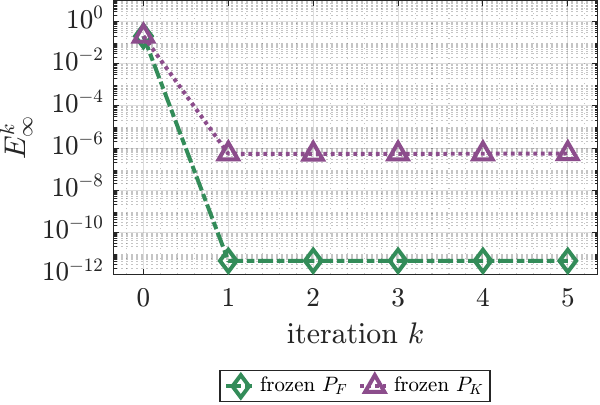}
  \caption{Error histories over five iterations after restarting
  projection-based PP-PC from \(\mathbf U^0\) with the discrepancy-form
  propagator and frozen projections \(P_F\) (mixed Fourier) and \(P_K\)
  (solution-snapshot) for the wave (left) and Schr\"odinger (right) problems.}
  \label{fig:linear-posterior}
\end{figure}

Within each problem, the computed cyclic majorants are comparable, whereas the
unresolved initial error for \(P_F\) is more than five orders of magnitude
smaller than for \(P_K\).  Consistently, \(P_F\) produces one-step errors of
order \(10^{-12}\), while \(P_K\) produces errors of order \(10^{-5}\) and
\(10^{-7}\) for the wave and Schr\"odinger problems, respectively.  The
subsequent fixed-projection histories preserve this separation.  Within the
evaluated estimate, the unresolved factor therefore provides the dominant
separation between the two projections.  Thus, in these tests,
the reported Fourier-aware construction yields a more refined projection space
for the error-relevant components.

\subsection{Nonlinear problems}
\label{subsec:numerics-nonlinear}

Both nonlinear problems have period \(T=1\) and periodic boundary conditions
on \([0,1)\).  We use second-order centered differences on \(32\) grid points
(\(d=64\)), \(N=16\), and \(K=8\).  The fine propagator is classical
fourth-order Runge--Kutta with \(\delta t=1/512\).  Writing the semi-discrete
system as \(\mathbf w'=A\mathbf w+\mathbf g(\mathbf w,t)\), the coarse
propagator is an implicit--explicit (IMEX) Euler step
\cite{ascher1997imexRK,ascher1995imex}
\begin{equation}
  \label{eq:numerics-semi-implicit}
  \cG_n(\mathbf w)
  =
  \left(\Id-\Delta T\,A\right)^{-1}
  \left(\mathbf w+\Delta T\,\mathbf g(\mathbf w,T_{n+1})\right),
  \qquad \Delta T=1/16.
\end{equation}
Thus the stiff linear part is implicit and the nonlinear part explicit, with
one fixed-matrix solve per coarse step.

The first problem is the periodically forced \emph{Stuart--Landau}
reaction--diffusion system \cite{aranson2002cgl}.  For \(w=p+\imath q\),
\begin{equation}
  \label{eq:numerics-stuart-landau}
  \begin{aligned}
    \partial_t w
    &=
    D\,\partial_{xx}w+(\lambda+\imath\omega)w
    -(1+\imath\beta)|w|^2w+f(x,t),\\
    f(x,t)
    &=
    0.5\left(\sin2\pi x+0.35\sin4\pi x\right)\cos2\pi t,
  \end{aligned}
\end{equation}
where \(D=0.05\), \(\lambda=0.5\), \(\omega=2\pi\), and \(\beta=1\).

The second problem is the periodically forced \emph{Brusselator}
\cite{prigogine1968brusselator}, with \(a=1\), \(b=3\),
\(D_1=10^{-2}\), \(D_2=5\times10^{-3}\), and the same forcing \(f\):
\begin{equation}
  \label{eq:numerics-brusselator}
  \begin{aligned}
    \partial_t u
    &=
    D_1\,\partial_{xx}u+a-(b+1)u+u^2v+f(x,t),\\
    \partial_t v
    &=
    D_2\,\partial_{xx}v+bu-u^2v.
  \end{aligned}
\end{equation}
Although the forcing contains only the temporal frequency \(\ell=1\), the
nonlinear terms generate higher harmonics.  Both tests therefore use the fixed
7-mode band
\(\mathcal I_3=\{0,1,2,3,N-3,N-2,N-1\}\).

The two problems exercise the two stability mechanisms evaluated below.  For
Stuart--Landau, the measured linearization of \(\cH_{P^k,n}\) along the fine
reference is contractive, so its evaluated bound uses the loop-gain
specialization associated with
\cref{rem:contraction-recovers-nondegeneracy}.  For the Brusselator, the
corresponding linearization is noncontractive, so its evaluated bound uses the
general cyclic-inverse form.  The local nondegeneracy condition of
\cref{ass:local-nonlinear-nondegeneracy} and the smallness condition
\eqref{eq:nonlinear-smallness} are examined below.  The
cyclic-inverse analysis therefore covers noncontractive settings in which the
linearized cyclic operator remains invertible, beyond the contraction-based
stability mechanism used in earlier quantitative nonlinear PP-PC estimates
\cite{gander2013periodicParareal}.

\paragraph{Convergence comparison}

\Cref{fig:nonlinear-convergence,tab:nonlinear-summary} show
that Fourier-aware PP-PC reaches \(E_\infty^k\le10^{-10}\) first, at \(k=2\)
for Stuart--Landau and \(k=4\) for the Brusselator, and is the only method to
do so in both tests within \(K=8\).  Under the same rank-revealing procedure,
its maximal projection-space dimension is smaller than that of
Krylov-enhanced PP-PC in both tests.  The panels show the histories through
\(k=6\), while the table summarizes the complete \(K=8\) runs by their final
errors, threshold iterations, maximal dimensions, and mode counts.

\begin{figure}[t!]
  \centering
  \includegraphics[width=0.48\textwidth]{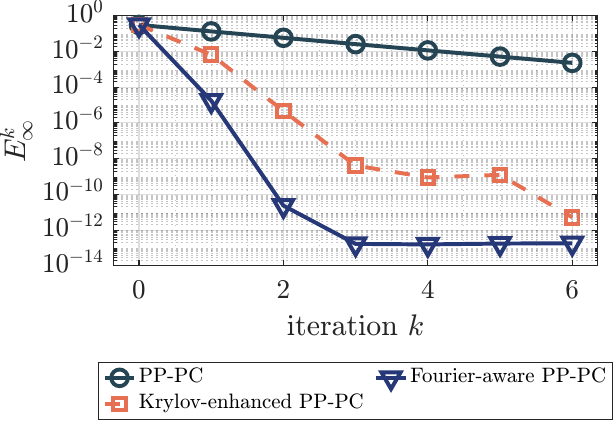}
  \hfill
  \includegraphics[width=0.48\textwidth]{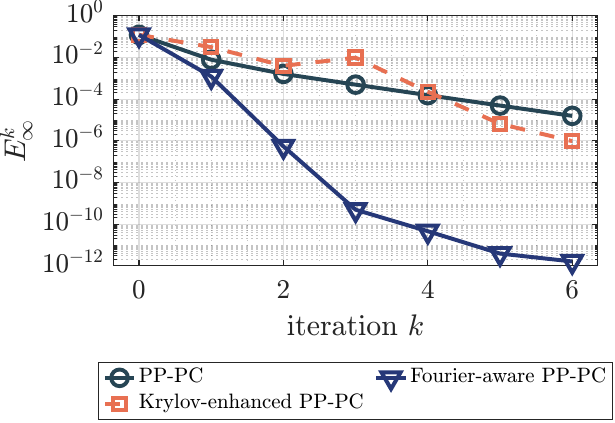}
  \caption{Error histories \(E_\infty^k\) of PP-PC,
  Krylov-enhanced PP-PC, and Fourier-aware PP-PC for the forced Stuart--Landau problem
  \eqref{eq:numerics-stuart-landau} (left) and the forced Brusselator
  problem \eqref{eq:numerics-brusselator} (right),
  shown through \(k=6\).}
  \label{fig:nonlinear-convergence}
\end{figure}

\begin{table}[t!]
  \caption{Convergence and projection-space dimension for the two nonlinear
  problems.}
  \label{tab:nonlinear-summary}
  \centering
  \footnotesize
  \begin{tabular}{@{}llcccc@{}}
    \toprule
    Problem & Method & \(E_\infty^{K}\) &
    \(E_\infty^k\le10^{-10}\) & \(\max_k\dim\) & \(|\mathcal I^k|\) \\
    \midrule
    Stuart--Landau \((K=8)\) & PP-PC                 & 4.5e-04 & --  & 0   & -- \\
                             & Krylov-enhanced PP-PC & 6.2e-13 & 6   & 23  & -- \\
                             & Fourier-aware PP-PC   & 1.8e-13 & 2   & 17  & 7 \\
    \midrule
    Brusselator \((K=8)\) & PP-PC                 & 1.7e-06 & --  & 0   & -- \\
                          & Krylov-enhanced PP-PC & 4.1e-09 & --  & 51  & -- \\
                          & Fourier-aware PP-PC   & 2.9e-13 & 4   & 36  & 7 \\
    \bottomrule
  \end{tabular}
\end{table}

\paragraph{Projection-level decomposition of the local one-step factor}
For \(P=P_{\mathcal I^k}\), set
\(Q=\Id-P_{\mathcal I^k}\) and define
\(m_Q^k:=\max_{0\le n<N}\norm{Q\mathbf e_n^k}_2\) and
\(m_P^k:=\max_{0\le n<N}\norm{P_{\mathcal I^k}\mathbf e_n^k}_2\); abbreviate
\(U_k^\perp:=U_k^\perp(P_{\mathcal I^k})\).
For the Fourier-aware projection, these definitions and
\cref{lem:tail-leak-projected-history} give
\[
  \begin{aligned}
  m_Q^k
  &=E_\infty^k\theta_Q^k(P_{\mathcal I^k})
  \le \widetilde m_Q^k
  :=\frac{\Tail_{\mathcal I^k}(\mathbf e^k)
  +\Leak_{\mathcal I^k}^{P_{\mathcal I^k}}(\mathbf e^k)}{\sqrt N}
  =E_\infty^k\Theta_Q^k,\\
  m_P^k&=E_\infty^k\Theta_P^k.
  \end{aligned}
\]
Thus, after multiplication by \(E_\infty^k\), the factor in
\cref{lem:local-nonlinear-onestep} has the three numerator contributions
\[
  \mu_k\,m_Q^k,
  \qquad
  \nu_k\,U_k^\perp m_P^k,
  \qquad
  \frac{\nu_k}{2}m_Q^k m_P^k.
\]
\Cref{tab:nonlinear-decomposition} reports this majorant, the three
unweighted factors, and the resulting \(E_\infty^{k+1}\).

In both problems, \(m_Q^k\) is the largest unweighted factor.
The measured weights preserve this dominance.  The discrepancy-curvature
factor \(m_Q^km_P^k\) is smaller and quadratic in the error, while
\(U_k^\perp m_P^k\) is negligible.  Over the reported steps,
\(\Theta_Q^k=\widetilde m_Q^k/E_\infty^k\le4.93\times10^{-2}\), so the
normalized tail--leak sum is small.  Since \(\widetilde m_Q^k\) majorizes the dominant unresolved
factor, these data provide a posteriori evidence that a small tail together
with a small leak is associated with the observed rapid next-step decay; the
complete estimate is evaluated below.

\paragraph{Evaluation of the nonlinear tail--leak estimate}
Thus \cref{thm:fourier-aware-nonlinear} replaces \(m_Q^k\) by its
tail--leak majorant \(\widetilde m_Q^k\) in the unresolved-error and
discrepancy-curvature contributions, while leaving the base-state contribution
unchanged.

\begin{table}[H]
  \caption{Measured tail--leak majorant, projection-level factors, and
  next-step error.}
  \label{tab:nonlinear-decomposition}
  \centering
  \small
  \begin{tabular}{lccccccc}
    \toprule
    Problem & \(k\) & \(E_\infty^k\) &
    \(\widetilde m_Q^k\) &
    \(m_Q^k\) &
    \(U_k^\perp m_P^k\) & \(m_Q^k m_P^k\) &
    \(E_\infty^{k+1}\) \\
    \midrule
    Stuart--Landau & 0 & 3.1e-01 & 6.8e-05 & 6.7e-05 & 4.9e-11 & 2.1e-05 & 1.8e-05 \\
                   & 1 & 1.8e-05 & 6.2e-10 & 1.6e-10 & 7.5e-17 & 2.8e-15 & 2.4e-11 \\
                   & 2 & 2.4e-11 & 1.2e-12 & 9.4e-13 & 2.1e-23 & 2.2e-23 & 1.7e-13 \\
    \midrule
    Brusselator & 0 & 1.2e-01 & 3.5e-03 & 3.4e-03 & 1.2e-09 & 4.2e-04 & 1.2e-03 \\
                & 1 & 1.2e-03 & 1.5e-06 & 1.4e-06 & 7.0e-12 & 1.8e-09 & 5.4e-07 \\
                & 2 & 5.4e-07 & 2.2e-09 & 2.0e-09 & 2.7e-17 & 1.1e-15 & 4.9e-10 \\
    \bottomrule
  \end{tabular}
\end{table}

The corresponding diagnostic weighted sum is
\[
  \mathcal S_{k,\mathrm F}
  :=\hat\mu\,\widetilde m_Q^k
  +\hat\nu\,U_k^\perp m_P^k
  +\frac{\hat\nu}{2}\widetilde m_Q^k m_P^k.
\]

For Stuart--Landau, the computed finite-difference (FD) reference-Jacobian norm
\[
  \hat\alpha_k
  :=\max_n\norm{J_{\cH_{P_{\mathcal I^k},n}}^{\rm FD}
  (\mathbf u_n^\star)}_2
  \approx0.956<1,
\]
giving the loop-gain specialization
\begin{equation}
  \label{eq:numerics-estimated-nl-bound}
  B_{k+1,\mathrm F}^{\mathrm{nl}}
  :=\frac{\mathcal S_{k,\mathrm F}}
  {(1-\hat\alpha_k)-\tfrac12\hat\kappa_kR}.
\end{equation}
For the Brusselator, whose reference linearization is noncontractive, we use
the general cyclic-inverse form
\begin{equation}
  \label{eq:numerics-nk-nl-bound}
  B_{k+1,\mathrm F}^{\mathrm{nl}}
  :=\frac{\overline{\widehat\Gamma}_k}
  {1-\tfrac12\overline{\widehat\Gamma}_k\hat\kappa_kR}
  \mathcal S_{k,\mathrm F}.
\end{equation}

Here \(\overline{\widehat\Gamma}_k\) is defined analogously by
the block-row formula in \eqref{eq:block-row-majorant}, using the assembled
finite-difference cyclic matrix.  With exact Jacobians, the same construction majorizes
\(\widehat\Gamma_k=\norm{\widehat{\cC}_{P_{\mathcal I^k}}^{-1}}_{\infty\to\infty}\);
the prefactor is increasing in this norm while its denominator is positive,
so the majorization preserves the exact bound.

We take \(R=E_\infty^0\).  The Stuart--Landau denominator is positive, while
for the Brusselator \(\overline{\widehat\Gamma}_k\in[62,66]\) and
\(\tfrac12\overline{\widehat\Gamma}_k\hat\kappa_kR
\in[0.50,0.66]\); moreover,
\(E_\infty^{k+1}\le R\) holds throughout both runs.  All
diagnostic Jacobians use central differences with step
\(10^{-6}(1+\norm{\mathbf x}_\infty)\).  Those at \(\mathbf u^\star\) give
\(\hat\alpha_k\) and the cyclic matrix used to form
\(\overline{\widehat\Gamma}_k\), while \(\hat\mu\) maximizes
discrepancy-Jacobian norms at \(\mathbf u^\star\) and \(\mathbf U^0\).  The
estimate \(\hat\nu\) uses Jacobian difference quotients between these two
histories and along two directions generated with seed \(0\) and scaled to
radius \(E_\infty^0\) at each time point; the same directional sampling of the
Jacobians of \(\cH_{P_{\mathcal I^k},n}\) gives the sampled curvature
\(\hat\kappa_k\).
These finite-difference and sampled quantities are diagnostic estimates rather
than certified bounds for the exact local constants.

\Cref{fig:nonlinear-bound} shows that both evaluated estimates lie above the
realized errors and follow their decay.  Relative to the projection-level
estimate, the tail--leak majorization adds, over the evaluated iterations, at
most factors \(3.97\) and \(1.45\) for Stuart--Landau and the Brusselator,
respectively.  Thus
\cref{tab:nonlinear-decomposition} isolates the local projection mechanism,
while \cref{fig:nonlinear-bound} evaluates the corresponding
Fourier tail--leak estimate of \cref{thm:fourier-aware-nonlinear}, including the noncontractive
regime.

\begin{figure}[t!]
  \centering
  \includegraphics[width=0.48\textwidth]{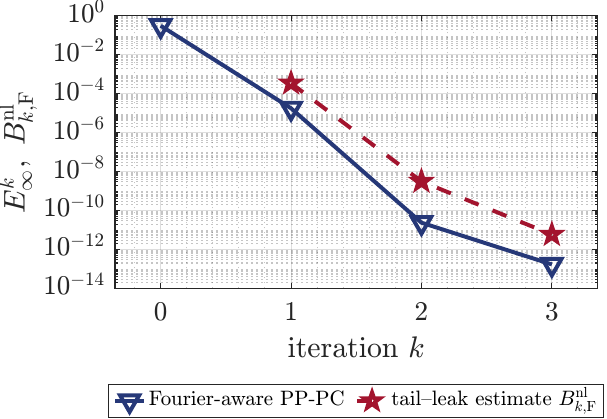}
  \hfill
  \includegraphics[width=0.48\textwidth]{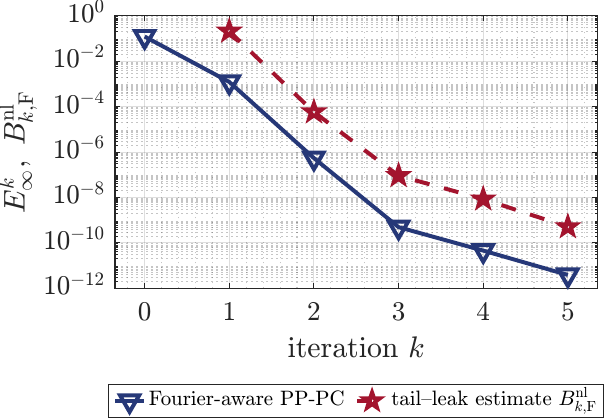}
  \caption{Measured errors \(E_\infty^k\) and evaluated nonlinear tail--leak
  estimates \(B_{k,\mathrm F}^{\mathrm{nl}}\) of
  \cref{thm:fourier-aware-nonlinear}.  Left: the loop-gain specialization
  \eqref{eq:numerics-estimated-nl-bound} for Stuart--Landau.  Right: the
  cyclic-inverse form \eqref{eq:numerics-nk-nl-bound} for the Brusselator.
  The value \(B_{0,\mathrm F}^{\mathrm{nl}}\) is not defined.}
  \label{fig:nonlinear-bound}
\end{figure}

\subsection{Summary and discussion}
\label{subsec:numerics-summary}

\Cref{tab:linear-summary,tab:nonlinear-summary} show that Fourier-aware PP-PC
reaches \(E_\infty^k\le10^{-10}\) in the fewest outer iterations in all four
examples.  Its maximal projection-space dimension is smaller than that of
Krylov-enhanced PP-PC in three examples.  For the wave problem it is larger,
but the tolerance is reached at \(k=8\) rather than \(k=15\)
(\cref{fig:linear-convergence}).  The frozen-projection experiment illustrates
the dependence on \(P\) in \cref{thm:projection-level-linear-estimate}: at the
same history depth, the mixed Fourier spaces leave smaller unresolved errors
and produce smaller one-step errors under comparable computed cyclic majorants
(\cref{tab:linear-posterior,fig:linear-posterior}).

For the linear problems, the measured tail--leak quantities and evaluated bound
are consistent with \cref{thm:linear-tail-leak-estimate}
(\cref{tab:linear-tail-leak,fig:linear-bound}).  For the nonlinear problems,
\cref{tab:nonlinear-decomposition} identifies \(m_Q^k\) as the largest
unweighted projection-level factor.  Over the reported steps,
\(\Theta_Q^k\le4.93\times10^{-2}\).  \Cref{fig:nonlinear-bound} evaluates the
nonlinear tail--leak estimate using diagnostic estimates of
the local constants for the Stuart--Landau and Brusselator
cases, whose computed finite-difference reference linearizations are contractive
and noncontractive, respectively.  Taken
together, these results indicate that the mixed Fourier construction yields a
more refined projection space in the reported tests: it is better aligned with
error-relevant directions but not always smaller.

\section{Conclusion and future work}
\label{sec:conclusion}

We introduced Fourier-aware projection-based PP-PC, a
projection-based periodic \emph{parareal} method that combines a
discrepancy-based correction scheme with a Fourier-aware projection space.
We developed a convergence analysis for general nonlinear time-periodic
problems, in which a temporal tail--leak bound controls both the
unresolved-error term and the nonlinear remainder in a local one-step estimate,
and showed that the
linear case reduces to global bounds under weaker
assumptions.  We demonstrated on linear and nonlinear problems that
Fourier-aware PP-PC converges faster than Krylov-enhanced PP-PC, with the
measured errors consistent with the analysis.

Our comparisons measure outer-iteration convergence, as in
\cite{gander2013periodicParareal,song2024krylovpppc}.  Since the corrected
periodic solve is more expensive per iteration than the coarse solve of the
original PP-PC, a work-normalized timing study and dedicated fast solvers for it are
subjects of further study.

\newpage
\section*{Acknowledgments}
This work was supported by the National Key R\&D Program of China under Grant
Nos. 2020YFA0711900 and 2020YFA0711902.  During the preparation of this
manuscript, the authors used OpenAI's ChatGPT to improve its readability and assist
with the design of code for the numerical experiments.  The authors assume
responsibility for all content.

\bibliographystyle{siamplain}
\bibliography{references}

\end{document}